\documentclass[11pt]{amsart}

\usepackage[french,british]{babel}
\usepackage{amsthm}
\usepackage{amsmath}
\usepackage{amssymb}
\usepackage{mathrsfs}
\usepackage{centernot}   
\usepackage{graphicx}        % pacchetto grafico standard LaTeX 
\usepackage{multicol}        % usato per l'indice a due colonne
\usepackage[all,ps]{xy}
\usepackage{layout}
\usepackage{booktabs}

\newcommand{\CC}{{\mathbb C}}

\newcommand{\aff}{{\mathbb A}}

\newcommand{\cF}{{\mathscr F}}

\newcommand{\cI}{{\mathscr I}}

\newcommand{\cN}{{\mathscr N}}

\newcommand{\cO}{{\mathscr O}}

\newcommand{\cR}{{\mathscr R}}
\newcommand{\cS}{{\mathscr S}}
\newcommand{\cT}{{\mathscr T}}
\newcommand{\cU}{{\mathscr U}}
\newcommand{\cV}{{\mathscr V}}

\newcommand{\cZ}{{\mathscr Z}}

\newcommand{\gm}{\mathfrak{m}}

\newcommand{\gS}{\mathfrak{S}}

\newcommand{\hra}{\hookrightarrow}

\newcommand{\la}{\langle}

\newcommand{\lra}{\longrightarrow}

\newcommand{\n}{\noindent}

\newcommand{\ov}{\overline}
\newcommand{\PP}{{\mathbb P}}
\newcommand{\QQ}{{\mathbb Q}}
\newcommand{\ra}{\rangle}

\newcommand{\wh}{\widehat}
\newcommand{\wt}{\widetilde}
\newcommand{\ZZ}{{\mathbb Z}}

\theoremstyle{plain}

\newtheorem{thm}{Theorem}[section]

\newtheorem{clm}[thm]{Claim}

\newtheorem{crl}[thm]{Corollary}
\newtheorem{hyp}[thm]{Hypothesis}
\newtheorem{lmm}[thm]{Lemma}
\newtheorem{prp}[thm]{Proposition}
\newtheorem{prp-dfn}[thm]{Proposition-Definition}

\theoremstyle{definition}

\newtheorem{dfn}[thm]{Definition}

\theoremstyle{remark}

\newtheorem{expl}[thm]{Example}

\newtheorem{rmk}[thm]{Remark}

\DeclareMathOperator{\CH}{CH}

\DeclareMathOperator{\cod}{cod}

\DeclareMathOperator{\DCH}{DCH}
\DeclareMathOperator{\Dec}{Dec}

\DeclareMathOperator{\divisore}{div}

\DeclareMathOperator{\mult}{mult}
\DeclareMathOperator{\NL}{NL}

\DeclareMathOperator{\Pic}{Pic}

\DeclareMathOperator{\sing}{sing}
\DeclareMathOperator{\Spec}{Spec}

\DeclareMathOperator{\supp}{supp}
\DeclareMathOperator{\Sym}{S}

\DeclareMathOperator{\Ver}{Vert}

\usepackage{ifthen}
\usepackage{rotating}
\usepackage{supertabular}
\newcommand{\cit}[1]{{\rm \textbf{#1}}}
\newcommand{\Ref}[2]{\cit{%
\ifthenelse{\equal{#1}{thm}}{Theorem}{}%
\ifthenelse{\equal{#1}{ass}}{Assumption}{}%
\ifthenelse{\equal{#1}{chp}}{Chapter}{}%
\ifthenelse{\equal{#1}{prp}}{Proposition}{}%
\ifthenelse{\equal{#1}{lmm}}{Lemma}{}%
\ifthenelse{\equal{#1}{cnj}}{Conjecture}{}%
\ifthenelse{\equal{#1}{crl}}{Corollary}{}%
\ifthenelse{\equal{#1}{dfn}}{Definition}{}%
\ifthenelse{\equal{#1}{expl}}{Example}{}%
\ifthenelse{\equal{#1}{hyp}}{Hypothesis}{}%
\ifthenelse{\equal{#1}{rmk}}{Remark}{}%
\ifthenelse{\equal{#1}{clm}}{Claim}{}%
\ifthenelse{\equal{#1}{prb}}{Problem}{}%
\ifthenelse{\equal{#1}{exe}}{Exercise}{}%
\ifthenelse{\equal{#1}{qst}}{Question}{}%
\ifthenelse{\equal{#1}{sec}}{Section}{}%
\ifthenelse{\equal{#1}{subsec}}{Subsection}{}%
\ifthenelse{\equal{#1}{subsubsec}}{Subsubsection}{}%
\ifthenelse{\equal{#1}{univ}}{Universal Property}{}%
\ifthenelse{\equal{#1}{trm}}{Terminology}{}%
\ifthenelse{\equal{#1}{tbl}}{Table}{}%
\  \ref{#1:#2}%
}}

\usepackage{multirow}

 \makeindex
 
\begin{document}
 \title{Decomposable cycles and Noether-Lefschetz loci}
 \author{Kieran G. O'Grady\\\\
\lq\lq Sapienza\rq\rq Universit\`a di Roma}
\dedicatory{Alla piccola Titti}
\date{June 26  2015}
\thanks{Partially supported by PRIN 2013, the Giorgio and Elena Petronio Fellowship Fund, the Giorgio and Elena Petronio Fellowship Fund II, the Fund for Mathematics of the IAS}
  \maketitle
%\tableofcontents
%
\section{Introduction}
Let $X$ be a smooth complex surface: a rational equivalence class of $0$-cycles on $X$ is \emph{decomposable}  if it is the intersections of two divisor classes. Let $\DCH_0(X)\subset\CH_0(X)$  be 
the subgroup   generated by decomposable $0$-cycles.   Beaville and Voisin~\cite{beauvoisin} proved that if $X$ is a $K3$ surface then $\DCH_0(X)\cong\ZZ$.    What can be said of the group  $\DCH_0(X)$ in general?  An irregular surface $X$ with non-zero map $\bigwedge^2 H^0(\Omega^1_X)\to H^0(\Omega^2_X)$ provides an example with group of decomposable $0$-cycles that is not finitely generated, even after tensorization with $\QQ$. Let us assume that $X$ is a  regular surface: then  $\DCH_0(X)$ is finitely generated because $\CH^1(X)$ is finitely generated, and we may ask for its the rank.  Blowing up  regular surfaces with non-zero geometric genus at  $(r-1)$ generic points, one gets  examples of regular surfaces with  $\DCH_0(X)$ of   rank at least $r$ (see Example 1.3 b) of~\cite{beausplit}). 
What about a less artificial class of surfaces, such as (smooth)  surfaces  in $\PP^3$?  If the rank of $\DCH_0(X)$ is to be larger than $1$ then the rank of $\CH^1(X)$ must be larger than $1$, but the latter condition is not sufficient, for example curves on $X$ whose canonical line-bundle is a (fractional) power of the hyperplane bundle do not   increase the rank of $\DCH_0(X)$, see~\Ref{subsec}{relstan}. The papers~\cite{shioda,boisarti} provide examples of  smooth surfaces in $\PP^3$ with Picard  group of large rank and generated by lines: it follows that the group spanned by decomposable $0$-cycles of such surfaces  has rank $1$. 
 On the other hand Lie Fu proved that  there exist degree-$8$ surfaces  $X\subset\PP^3$ such that  $\DCH_0(X)$ has rank at least $2$, see 
 1.4 of~\cite{fudecomp}.
In the present paper we will prove the result below.
\begin{thm}\label{thm:zorba}
There exist smooth surfaces $X\subset\PP^3$ of degree $d$ such that the rank of $\DCH_0(X)$ is at least $\lfloor \frac{d-1}{3}\rfloor$. 
\end{thm}
In particular the rank of the group of decomposable $0$-cycles of a smooth surface in $\PP^3$ can be arbitrarily large. 

Let us explain the main ideas that go into the proof of~\Ref{thm}{zorba}.   Let $C=C_1\cup\ldots\cup C_n$  be the disjoint union  of  smooth irreducible curves  $C_j\subset \PP^3$. Suppose that 
 if $d\gg 0$, and that the curves $C_j$ are not rationally canonical, i.e.~there exists $e\in\ZZ$ such that $K_{C_j}^{\otimes m}\cong \cO_{C_j}(e)$ only for $m=0$;   we prove that for a very general 
 smooth $X\in|\cI_C(d)|$, the classes  $c_1(\cO_X(1))^2,C_1\cdot C_1,\ldots,C_n\cdot C_n$ in $\CH_0(X)$ are linearly independent. We argue as follows. Assume that they are \emph{not} linearly independent for $X$ very general; then there exists a non-zero $(a,r_1,\ldots,r_n)\in\ZZ^{n+1}$ such that   
\begin{equation}\label{relazione}
a c_1(\cO_X(1))^2+r_1 c_1(\cO_X(C_1))^2+\ldots+ r_n c_1(\cO_X(C_n))^2=0
\end{equation}
 for all  smooth $X\in|\cI_C(d)|$. 
Now let $\pi\colon W\to\PP^3$ be the blow up of $C$,  let $E$ be the exceptional divisor of $\pi$, and $E_j$ be the component of $E$ mapping to $C_j$. Let $\cS\subset W\times\Lambda(d)$ be the universal surface parametrized by $\Lambda(d)$. We let $p_W\colon\cS\to W$ and  $p_{\Lambda(d)}\colon\cS\to \Lambda(d)$ be the projection maps. There is a natural identification $\Lambda(d)=|\cI_C(d)|$, and the generic $S\in \Lambda(d)$  is isomorphic 
to the corresponding  $X\in |\cI_C(d)|$.  Since~\eqref{relazione} holds for all smooth $X$, an application of  the spreading principle shows that the class
\begin{equation}\label{pamphili}
p_W^{*}(a \pi^{*} c_1(\cO_\PP^3(1))^2+r_1 c_1(\cO_W(E_1))^2+\ldots+ r_n c_1(\cO_W(E_n))^2)\in\CH^2(\cS)
\end{equation}
is \emph{vertical}, i.e.~it is represented by a sum of codimension-$2$ subvarieties $\Gamma_i\subset \cS$ such that 
\begin{equation}\label{diminuisce}
\dim p_{\Lambda(d)}(\Gamma_i)<\dim\Gamma_i.
\end{equation}
We prove that if the class in~\eqref{pamphili} is vertical, then $0=a=r_1=\ldots=r_n$. The key result that one needs is a Noether-Lefschetz Theorem for surfaces belonging to an integral  codimension-$1$ closed subset $A\in\Lambda(d)$. More precisely one needs to prove that the following hold:
\begin{enumerate}
\item
 If the generic $S\in A$ is isomorphic to $\pi(S)\subset \PP^3$, i.e.~$S$ contains no fiber of $\pi\colon W\to \PP^3$ over $C$, then $\CH^1(S)$ is generated (over $\QQ$) by $\pi^{*} c_1(\cO_{\PP^3}(1))|_S,c_1(\cO_S(E_1),\ldots,c_1(\cO_S(E_n)$.
\item
If the generic $S\in A$  contains a fiber $R$ of $\pi\colon W\to \PP^3$ over $C$, necessarily unique by genericity of $S$, then  $\CH^1(S)$ is generated (over $\QQ$) by the classes listed in Item~(1), together with $c_1(\cO_S(R))$.
\end{enumerate}
The reason why such a Noether-Lefschetz Theorem  is needed is the following. Let  $\Gamma_i\subset\cS$ be a codimension-$2$ subvariety such that~\eqref{diminuisce} holds, and assume that the generic fiber of $\Gamma_i\to p_{\Lambda(d)}(\Gamma_i)$ has dimension $1$; then $A:=p_{\Lambda(d)}(\Gamma_i)$ is an integral closed codimension-$1$ subset of $\Lambda(d)$, and the restriction of $\Gamma_i$ to the surface $S_t$ parametrized by $t\in A$ is a divisor on $S_t$. Thus we are lead to prove the above Noether-Lefschetz result.  There is a substantial literature on Noether-Lefschetz, but we have not  found a result taylor made for our needs. A criterion of K.~Joshi~\cite{joshi} is very efficient in disposing of \lq\lq most\rq\rq\ choices of 
a codimension-$1$ closed subset $A\in\Lambda(d)$. We deal with the remaining cases by appealing to the Griffiths-Harris approach to Noether-Lefschetz~\cite{grihar}, as further developped by Lopez~\cite{lopez} and Brevik-Nollet~\cite{brenol}.

The paper is organized as follows. In~\Ref{sec}{sezuno} we consider a smooth $3$-fold $V$ with trivial Chow groups,  an ample  divisor  $H$ on $V$ and  surfaces in the linear system  $|\cI_C(H)|$,  where  $C=C_1\cup\ldots\cup C_n$  is the disjoint union  of a fixed collection of smooth irreducible curves  $C_i\subset V$. 
We prove that if the curves $C_i$ are not rationally canonical, and a suitable Noether-Lefschetz Theorem holds, then the classes of $C^2_1,\ldots,C^2_n$ on a very general  $X\in |\cI_C(H)|$ are linearly independent, and they span a subgroup intersecting trivially  the image of $\CH^2(V)\to \CH^2(X)$.  
    In~\Ref{sec}{oltrelopez}   we prove the required Noether-Lefschetz Theorem for $V=\PP^3_{\CC}$.
    In~\Ref{sec}{chiudo}   we prove~\Ref{thm}{zorba} by combining the main results of~\Ref{sec}{sezuno}  and~\Ref{sec}{oltrelopez}.

\bigskip
\n
{\bf Conventions and notation:} We work over $\CC$.  Points are closed points.  

Let $X$ be a variety: \lq\lq If $x$ is  a generic point of $X$, then...\rq\rq\ is shorthand for \lq\lq There exists an open dense $U\subset X$ such that if $x\in U$ then...\rq\rq. Similarly the expression \lq\lq If $x$ is  a very general point of $X$, then...\rq\rq\ is shorthand for \lq\lq There exists a countable collection of closed nowhere dense $Y_i\in X$  such that if $x\in (X\setminus \bigcup_i Y_i)$ then...\rq\rq. 

From now on we will  denote  by  $\CH(X)$ the group of rational equivalence classes of cycles with \emph{rational} coefficients.  Thus if  $Z_1,Z_2$  are cycles  on $X$ then $Z_1\equiv Z_2$  means that for some non-zero integer $\ell$ the cycles $\ell Z_1$,  $\ell Z_2$ are integral  and rationally equivalent. If $Z$ is a cycle on $X$ we will often use the same symbol (i.e.~$Z$) for the rational equivalence class represented by $Z$.

\medskip
\n
{\bf Acknowledgements:} It is a pleasure to thank Angelo Lopez for useful exchanges on Noether-Lefschetz results.
\section{The family of surfaces containing  given curves}\label{sec:sezuno}
\setcounter{equation}{0}
\subsection{Threefolds with trivial Chow groups}
Throughout the paper $V$ is an integral smooth projective threefold.
\begin{hyp}\label{hyp:ipovu}
The cycle class map $cl\colon \CH(V)\lra H(V;\QQ)$ is an isomorphism.
\end{hyp} 
The archetypal  such  $V$ is $\PP^3$. A larger class of examples is given by   $3$-folds with an algebraic cellular decomposition (see Ex.~1.9.1 of~\cite{fulton}), and conjecturally the above assumption is equivalent to vanishing of $H^{p,q}(V)$ for $p\not=q$.  An integral smooth projective threefold has \emph{trivial} Chow group   if~\Ref{hyp}{ipovu} holds.
\begin{clm}\label{clm:ciaodue}
Let $V$ be as above, in particular it has trivial Chow group. The natural map 
\begin{equation}\label{simdiv}
\Sym^2\CH^1(V)\lra\CH^2(V)
\end{equation}
is surjective.
\end{clm}
\begin{proof}
The natural map  $\Sym^2 H^2(V;\QQ)\to H^4(V;\QQ)$ is surjective by Hard Lefscehtz. The claim follows because of~\Ref{hyp}{ipovu}.
\end{proof}
\subsection{Standard relations}\label{subsec:relstan}
Let $V$ be an integral smooth projective $3$-fold with trivial Chow group. 
Let $X\subset V$ be a closed surface, and $i\colon X\hra V$ be the inclusion map. Let $\cR^s(X)\subset\CH^s(X)$ be the image of the restriction map
\begin{equation}
\begin{matrix}
\CH^s(V) & \lra & \CH^s(X) \\
\xi & \mapsto & i^{*}\xi
\end{matrix}
\end{equation}
Notice that $\cR^2(X)\subset\DCH_0(X)$ by~\Ref{clm}{ciaodue}.
Suppose that  $C\subset X$ is an integral smooth curve. We will assume that $C\cdot C$ makes sense in $\CH_0(X)$, for example that will be the case if $X$ is $\QQ$-factorial. 
We will list elements of the kernel of the map
\begin{equation}
\begin{matrix}
\cR^2(X)\oplus \cR^1(X)\oplus \cR^0(X) & \lra & \DCH_0(X) \\
(\alpha,\beta,\gamma) & \mapsto & \alpha+C\cdot \beta+\gamma\cdot C\cdot C
\end{matrix}
\end{equation}
Let $j\colon C\hra V$ be the inclusion map. By  Cor.~8.1.1 of~\cite{fulton}  the following relation holds in $\CH_0(X)$:
\begin{equation}\label{tiraspingi}
i^{*}(j_{*}[C])=C\cdot c_1(\cN_{X/\PP^3})=C\cdot i^{*}\cO_V(X).
\end{equation}
Thus
\begin{equation}\label{guardaunpo}
\alpha_C-C\cdot i^{*}\cO_V(X)=0,
\end{equation}
where  $\alpha_C:=i^{*}(j_{*}C)\in  \cR^2(X)$.
Equation~\eqref{guardaunpo} is the \emph{first standard relation}. 

Now suppose that there exists $\xi\in\CH^1(V)$ such that
\begin{equation}\label{canesterno}
c_1(K_C)=\xi|_C.
\end{equation}
(Recall that Chow groups are with $\QQ$-coefficients, thus~\eqref{canesterno} means that there exists an integer $n>0$ such  that $K_C^{\otimes n}$ is the pull-back of a line-bundle on $V$.)
By adjunction  for $X\subset V$ and for $C\subset X$,
\begin{equation}\label{ernia}
C\cdot C +C\cdot(i^{*}K_V+i^{*}\cO_X(X))\equiv C\cdot i^{*}\xi.
\end{equation}
Thus there exists $\beta_C\in \cR^1(X)$  such that 
\begin{equation}\label{princeton}
\beta_C\cdot C-C\cdot C=0.
\end{equation}
The above is the \emph{second standard relation} (it holds assuming~\eqref{canesterno}). 
\begin{expl}
Let $V=\PP^3$, let $X\subset\PP^3$ be a smooth surface of degree $d$, and let $C\subset X$ be a smooth curve. The subgroup of  $\CH_0(X)$ spanned by    intersections of linear combinations of $H:=c_1(\cO_X(1))$ and $C$  
has rank  at most $2$. In fact the first standard relation reads $d C\cdot H=(\deg C)H\cdot H$. Suppose that   $c_1(K_C)=m C\cdot H$, where $m\in\QQ$. With this hypothesis, the second standard relation reads $C\cdot C=(m+4-d)C\cdot H$, and hence $C\cdot C, C\cdot H, H\cdot H$ span a rank-$1$ subgroup. In particular a curve of genus $0$ or $1$ does not add anything to the rank of $\DCH_0(X)$.
\end{expl}
\subsection{Surfaces containing disjoint curves}\label{subsec:supcurve}
Let $V$ be a smooth projective $3$-fold with trivial Chow group and $C_1,\ldots,C_n\subset V$ be pairwise disjoint  integral smooth projective curves. Let $C:=C_1\cup\ldots\cup C_n$ and let $\pi\colon W\to V$ be  the blow-up of $C$. Let $E$ be the exceptional divisor of $\pi$, and let $E_j$, 
 for $j\in\{1,\ldots,n\}$,  be the irreducible component of $E$ mapping to $C_j$. Let $H$ be an ample divisor on $V$. For $j\in\{1,\ldots,n\}$ we let
\begin{equation}
\Sigma_j:=\{S\in |\pi^{*}H-E| \mid \text{$\pi(S)$ is singular at some point of $C_j$}\},\qquad \Sigma:=\cup_{j=1}^n \Sigma_j.
\end{equation}
Let $S\in |\pi^{*}H-E|$, and  let $X:=\pi(S)$. Then $S\in\Sigma_j$ if and only if $S$ contains one (at least) of the fibers of $E_j\to C_j$, or, equivalently, 
   the map $S\to X$ given by restriction of $\pi$  is \emph{not} an isomorphism over $C_j$. We will always assume that $(\pi^{*}H-E)$ is very ample on $W$; with this hypothesis   $\Sigma_j$ is irreducible of codimension $1$, or empty (compute the codimension of the loci of  $S\in |\pi^{*}(H)-E|$ which contain one or two fixed  fibers  of $E_k\to C_k$).  
Suppose that  $H$ is  sufficiently ample:  then, in addition,   if  $S\in\Sigma_k$ is generic the surface $X=\pi(S)$ is smooth except for one ODP (ordinary double point) belonging to $C_k$,  and  the set of reducible $S\in|\pi^{*}H-E|$ is of large codimension in $|\pi^{*}H-E|$. We will assume that both of these facts hold (but we do not assume that $H$ is \lq\lq sufficiently ample\rq\rq, because we want to prove effective results).
\begin{hyp}\label{hyp:dieci}
Let $C_1,\ldots,C_n\subset V$ and $H$  be as above, in particular $H$ is ample on $V$, and $(\pi^{*}H-E)$ is very ample on $W$. Suppose   that 
\begin{enumerate}
\item
 for $j\in\{1,\ldots,n\}$, and  $S\in\Sigma_j$   generic, the surface $\pi(S)$ is  smooth except for one ODP (ordinary double point) belonging to $C_j$, and 
\item
the set of reducible $S\in|\pi^{*}H-E|$ has codimension  at least $3$ in  $|\pi^{*}H-E|$. 
\end{enumerate}
\end{hyp}
Assume that~\Ref{hyp}{dieci} holds, and let $S\in\Sigma_j$ be  generic. Then there is a  unique singular point of $\pi(S)$, call it $x$, and  the line $\pi^{-1}(x)$ is contained in $S$.  
\begin{hyp}\label{hyp:ipnole}
Let $C_1,\ldots,C_n\subset V$ and $H$  be as above. Suppose that~\Ref{hyp}{dieci} holds, and that in addition the following  hold:
\begin{enumerate}
\item
If $A\subset|\pi^{*}H-E|$ is an integral closed codimension-$1$ subset, not equal to one of $\Sigma_1,\ldots,\Sigma_n$, and $S\in A$ is very general, the restriction map  $\CH^1(W)\to \CH^1(S)$ is surjective.
\item
For  $j\in\{1,\ldots,n\}$, $S\in\Sigma_j$  very general, and   $x$  the   unique singular point of $\pi(S)$  
 (an ODP belonging to $C_j$, by~\Ref{hyp}{dieci}),  
 $\CH^1(S)$  is generated by the image  of the restriction map  $\CH^1(W)\to \CH^1(S)$  together with the  class of $\pi^{-1}(x)$. 
\end{enumerate}
\end{hyp}
\begin{rmk}
Let $V=\PP^3$, and fix  $C_1,\ldots,C_n\subset \PP^3$. Let    $d\gg 0$, and $H\in|\cO_{\PP^3}(d)|$. If $S\in\Sigma_j$  is generic, then  $\pi^{-1}(x)$ does \emph{not} belong to the image   of the restriction map  $\CH^1(W)\to \CH^1(S)$.  
\end{rmk}
In the present section we will prove the following result.
\begin{prp}\label{prp:vaccini}
Let $C_1,\ldots,C_n\subset V$ and $H$  be as above, and assume that~\Ref{hyp}{ipnole} holds. Suppose also that for $j\in\{1,\ldots,n\}$ there does \emph{not} exist  $\xi\in\CH^1(V)$ such that  $c_1(K_{C_j})=\xi|_{C_j}$. (Recall that Chow groups are with coefficients in $\QQ$.)
Then for very general smooth $X\in|\cI_C(H)|$  the following hold:
\begin{enumerate}
\item
The map $\CH^2(V)\to \CH_0(X)$ is injective.
\item
 Let $\{\zeta_1,\ldots,\zeta_m\}$ be a basis of $\CH^1(V)$ (as $\QQ$-vector space). Suppose that for very general smooth $X\in|\cI_C(H)|$ 
\begin{equation*}
0=P(\zeta_1|X,\ldots,\zeta_m|X)+r_1 C_1^2+\ldots+r_n C_n^2,
\end{equation*}
where $P\in\QQ[x_1,\ldots,x_m]_2$ is a homogeneous quadratic polynomial. Then  $0=P(\zeta_1,\ldots,\zeta_m)=r_1=\ldots=r_n$. 
\end{enumerate}
\end{prp}
The proof of~\Ref{prp}{vaccini} will be given in~\Ref{subsec}{demolink}. Throughout the present section we let $V$, $C$, $W$, $E$ and $H$ be as above. 
\subsection{The universal surface}
Assume that~\Ref{hyp}{dieci}  holds. Let   
\begin{eqnarray}
\Lambda & := & |\pi^{*}(H)-E| \\
\cS & := & \{(x,S) \in W\times \Lambda \mid x\in S\}.
\end{eqnarray}
Let $p_W\colon\cS\to W$ and $p_{\Lambda}\colon\cS\to  \Lambda$ be the forgetful maps. Thus we have
\begin{equation}\label{triangolo} 
\xymatrix{
&  & \cS \ar[dl]_{p_W} \ar[dr]^{p_{\Lambda}} &   \\
 V &  \ar[l]_{\pi} W &   &    \Lambda}  
 \end{equation}
Let $N:=\dim\Lambda$.  Since  $(\pi^{*}(H)-E)$ is very ample it is globally generated and hence the map $p_W$ is a $\PP^{N-1}$-fibration. It follows that  $\cS$ is smooth and
\begin{equation}
\dim\cS=(N+2).
\end{equation}
\begin{dfn}
Let $\Ver^q(\cS/\Lambda)\subset\CH^q(\cS)$ be the subspace spanned  by rational equivalence classes of codimension-$q$ integral closed subsets $Z\subset \cS$ such that the dimension of $p_{\Lambda}(Z)$ is strictly smaller than the dimension of $Z$. 
\end{dfn}
The result below is an instance of the spreading principle.
\begin{clm}\label{clm:marmite}
Keep notation and assumptions as above, in particular~\Ref{hyp}{dieci}  holds. Let  $Q\in\QQ[x_1,\ldots,x_m,y_1,\ldots,y_n]_2$ be a homogeneous polynomial of degree $2$ and   let $\zeta_1,\ldots,\zeta_m\in\CH^1(V)$.   Then 
\begin{equation}\label{simonyi}
Q(\zeta_1|_X,\ldots, \zeta_m|_X, c_1(\cO_X(C_1)),\ldots,c_1(\cO_X(C_n)))=0
\end{equation}
  for all smooth  $X\in|\cI_{C}(H)|$ if and only if  
\begin{equation}\label{svaporaz}
 p_W^{*} Q(\pi^{*}\zeta_1,\ldots,\pi^{*}\zeta_m,E_1,\ldots,E_n)\in\Ver^2(\cS/\Lambda).
\end{equation}
\end{clm}
\begin{proof}
Suppose that~\eqref{simonyi} holds  for all smooth  $X\in|\cI_{C}(H)|$. Let $S\in\Lambda$ be generic,  $X:=\pi(S)$. Then $X$ is smooth and  the restriction of $\pi$ to $S$ defines an isomorphism $\varphi\colon S\overset{\sim}{\lra} X$, thus by our assumption
$$p_W^{*}Q(\pi^{*}\zeta_1,\ldots,\pi^{*}\zeta_m,E_1,\ldots,E_n)|_{S}=0.$$
Since $S$ is generic in $\Lambda$
it follows (see~\cite{blochsrini,voisinbook}) that there exists an open dense subset $\cU\subset\Lambda$   such that
\begin{equation}\label{mare}
p_W^{*}Q(\pi^{*}\zeta_1,\ldots,\pi^{*}\zeta_m,E_1,\ldots,E_n)|_{p_{\Lambda}^{-1}\cU}=0.
\end{equation}
(We recall that Chow groups are with rational coefficients, if we consider integer coefficients then~\eqref{mare} holds only up to torsion.) Let $B:=(\Lambda\setminus\cU)$. By the localization exact sequence
$$\CH_{N}(p_{\Lambda}^{-1}B)\lra \CH_{N}(\cS)\lra \CH_{N}(p_{\Lambda}^{-1}\cU)\lra 0$$
$p_W^{*}Q(\pi^{*}\zeta_1,\ldots,\pi^{*}\zeta_m,E_1,\ldots,E_n)$ is represented by an $N$-cycle supported on $p_{\Lambda}^{-1}B$, 
and hence~\eqref{svaporaz} holds because  $\dim B<N$.  Next, suppose  that~\eqref{svaporaz} holds. Then, by definition, the left-hand side of~\eqref{svaporaz} is represented by an $N$-cycle whose support is mapped by $p_{\Lambda}$ to a proper closed subset  $B\subset\Lambda$.  
Thus there exists an open dense $\cU\subset\Lambda$ such that the restriction of $p_W^{*} Q(\pi^{*}\zeta_1,\ldots,\pi^{*}\zeta_m,E_1,\ldots,E_n)$ to $p_{\Lambda}^{-1}\cU$ vanishes, e.g.~$\cU=\Lambda\setminus B$. By shrinking $\cU$ we may assume that  for $S\in\cU$ the surface $X:=\pi(S)$ is smooth. Let $S\in\cU$: then 
$0=p_W^{*} Q(\pi^{*}\zeta_1,\ldots,\pi^{*}\zeta_m,E_1,\ldots,E_n)|_S$, and since $X\cong S$ it follows that~\eqref{simonyi} holds for  $X=\pi(S)$. On the other hand the locus of  smooth  $X\in|\cI_{C}(H)|$  such that~\eqref{simonyi} holds is a countable union of closed subsets of $\Lambda_{sm}$ (the open dense subset of $\Lambda$ parametrizing smooth surfaces); since it contains an open dense subset of  $\Lambda_{sm}$ it is equal to  $\Lambda_{sm}$. 
\end{proof}
\subsection{The Chow groups of $\cS$ and $W$}
Assume that~\Ref{hyp}{dieci}  holds. Let $\xi\in\CH^1(\cS)$ be the pull-back of the hyperplane class on $\Lambda$ via the map $p_{\Lambda}$ of~\eqref{triangolo}. Since $p_W$ is the projectivization of a rank-$N$ vector-bundle on $W$  and $\xi$ restricts to the hyperplane class on each fiber of $p_W$ the Chow ring $\CH(\cS)$ is the $\QQ$-algebra generated by $p_W^{*}\CH(W)$ and $\xi$, with ideal of relations  generated by a single relation in codimension $N$.  We have $N\ge 3$ because$(\pi^{*}H-E)$ is very ample  by~\Ref{hyp}{dieci}; thus 
\begin{equation}\label{chowdue}
\begin{matrix}
\QQ\oplus \CH^1(W) \oplus \CH^2(W)  & \overset{\sim}{\lra} &  \CH^2(\cS) \\
(a_0,a_1,a_2) & \mapsto & a_0\xi^2+ p_W^{*}(a_1)\cdot\xi+ p_W^{*}(a_2)
\end{matrix}
\end{equation}
is  an isomorphism. The Chow groups  $\CH_q(W)$ are computed by first describing $\CH_q(E_j)$ for $j\in\{1,\ldots, n\}$, and then considering the localization exact sequence 
$$\bigoplus_j\CH_q(E_j)\lra\CH_q(W)\lra \CH_q(W\setminus (E_1\cup\ldots\cup E_n))\lra 0.$$
 One gets an isomorphism
\begin{equation}\label{gruppodiv}
\begin{matrix}
 \CH^1(V) \oplus \QQ^n  & \overset{\sim}{\lra} &  \CH^1(W) \\
(a,t_1,\ldots,t_n) & \mapsto & \pi^{*}a+\sum_{j=1}^n t_j E_j
\end{matrix}
\end{equation}
and an exact sequence
\begin{equation}
0\lra \CH^2(W)_{\hom}\lra \CH^2(W)\overset{cl}{\lra} H^4(W;\QQ)\lra 0
\end{equation}
where $\CH^2(W)_{\hom}$ is described as follows. 
Let $\rho_j\colon E_j\to C_j$ be the restriction of the blow-up map $\pi$, and $\sigma_j\colon E_j\hra W$ be the inclusion map;  
then we have an Abel-Jacobi isomorphism 
\begin{equation}\label{abjac}
\begin{matrix}
AJ\colon \CH^2(W)_{\hom} & \overset{\sim}{\lra} &  \bigoplus_{j=1}^n\CH_0(C_j)_{\hom} \\
\alpha & \mapsto & (\rho_{1,*}(\sigma_1^{*}\alpha),\ldots,\rho_{n,*}(\sigma_n^{*}\alpha)
\end{matrix}
\end{equation}
Let $AJ_j$ be the $j$-th component of the map $AJ$.
\begin{lmm}\label{lmm:conti}
Assume that~\Ref{hyp}{dieci} holds. Let 
$$\omega:=\pi^{*}\alpha+\sum_{j=1}^n E_j\cdot\pi^{*}\beta_j+\sum_{j=1}^n\gamma_j E_j\cdot E_j,$$
where $\alpha\in\CH^2(V)$, $\beta_j\in\CH^1(V)$, and $\gamma_j\in\QQ$ for $j\in\{1,\ldots, n\}$. Then the following hold:

\begin{enumerate}
\item
The cohomology class of $\omega$ vanishes if and only if 
\begin{equation}\label{sommaga}
\alpha=\sum_{j=1}^n\gamma_j C_j,
\end{equation}
and for all $j\in\{1,\ldots, n\}$ 
\begin{equation}\label{gradobi}
 \deg(\beta_j\cdot C_j)=-\gamma_j\deg(\cN_{C_j/V}).
\end{equation}
\item
Suppose that~\eqref{sommaga} and~\eqref{gradobi} hold. Then for  $j\in\{1,\ldots, n\}$
\begin{equation}\label{abejac}
AJ_j(\omega)=-\gamma_j c_1(\cN_{C_j/V})-c_1(\beta_j|_{C_j}).
\end{equation}
\end{enumerate}
\end{lmm}
\begin{proof}
Since the cohomology class map $cl\colon\CH^1(V)\to H^2(V;\QQ)$ is a surjection (by hypothesis), also the  cohomology class map $cl\colon\CH^1(W)\to H^2(W;\QQ)$ is surjective. By Poincar\`e duality it follows that 
$cl(\omega)=0$ if and only if 
$\deg(\omega\cdot\xi)=0$ for all $\xi\in\CH^1(W)$.
By~\eqref{gruppodiv} we must test $\xi=\pi^{*}\zeta$ with $\zeta\in\CH^1(V)$ and $\xi=E_i$ for $i\in\{1,\ldots, n\}$. We have
\begin{equation}\label{croazia}
\deg(\omega\cdot\pi^{*}\zeta)=
\deg\left(\left(\alpha-\sum_{j=1}^n\gamma_j C_j\right)\cdot\zeta\right).
\end{equation}
Since the cycle map $\CH^2(V)\to H^4(V;\QQ)$ is an isomorphism,  it follows that $\deg(\omega\cdot\pi^{*}\zeta)=0$ for all  $\zeta\in\CH^1(V)$ if and only if~\eqref{sommaga} holds.
Next, we test  $\xi=E_i$. In $\CH_0(C_i)$ 
\begin{equation}\label{vivachern}
\rho_{i,*}c_1(\cO_{E_i}(E_i))^2=-c_1(\cN_{C_i/V}),
\end{equation}
and hence
\begin{equation}\label{serbia}
\deg(\omega\cdot E_i)=
-\deg(\beta_i\cdot C_i)-\gamma_i\deg(\cN_{C_i/V}).
\end{equation}
This proves Item~(1).
Item~(2) follows from Equation~\eqref{vivachern}.  
\end{proof}
\begin{rmk}\label{rmk:secstand}
By~\Ref{lmm}{conti}  the kernel of the map
\begin{equation}
\begin{matrix}
\CH^2(V)\oplus\bigoplus_{k=1}^n \CH^1(V) \oplus\bigoplus_{k=1}^n\QQ & \lra & \CH^2(W) \\
(\alpha,\beta_1,\ldots\beta_n,\gamma_1,\ldots,\gamma_n) & \mapsto & 
\pi^{*}\alpha+\sum_{j=1}^n E_j\cdot\pi^{*}\beta_j+\sum_{j=1}^n\gamma_j E_j\cdot E_j
\end{matrix}
\end{equation}
is generated over $\QQ$ by the classes $E_j\cdot \pi^{*}\beta$, where  
$ \beta\in\CH^1(V)$ and  $\beta|_{C_j}=0$, 
together with the classes
\begin{equation}
 \pi^{*}[C_j]+E_j\cdot \pi^{*}\beta+ E_j\cdot E_j,
\end{equation}
where $\beta\in\CH^1(V)$,   $\deg(\beta\cdot C_j)=-\deg(\cN_{C_j/V})$, and 
\begin{equation}\label{quattrogatti}
-c_1(\cN_{C_j/V})-c_1(\beta|_{C_j})=0. 
\end{equation}
Next  notice that~\eqref{quattrogatti} holds  if  and only if $c_1(K_{C_j})$ is equal to the restriction of a class in $\CH^1(V)$ i.e.~\eqref{canesterno} holds. Assume that this is the case, and
that $X\in|\cI_{C}(H)|$ is a surface  smooth at all points of $C_j$. Let $S\in|\pi^{*}H-E|$ be the strict transform of $S$. 
Then  $S$ is isomorphic to $X$ over $C_j$, and restricting to $S$ the equation $\pi^{*}[C_j]+E_j\cdot \pi^{*}\beta+ E_j\cdot E_j=0$ we get 
  the second standard relation~\eqref{princeton}. 
\end{rmk}
\subsection{A vertical   cycle on $\cS$}
According to~\Ref{clm}{marmite}, for every codimension-$2$ relation that holds between  $\cO_X(C_{1}),\ldots,  \cO_X(C_{n})$ and  restrictions to $X$ of divisors on $V$, where $X$ is an arbitrary smooth member of $\in|\cI_C(H)|$, there is a polynomial in classes of $\pi^{*}\CH^1(V)$ and the classes of the exceptional divisors of $\pi$ which is \lq\lq responsible\rq\rq\ for the relation, i.e.~when we pull-it back to $\cS$ it is a vertical class. 
We have shown that $\pi^{*}[C_j]+E_j\cdot \pi^{*}\beta+ E_j\cdot E_j$ is the class responsible for the  second standard relation~\eqref{princeton}, 
see~\Ref{rmk}{secstand}, and in fact this class vanishes. In the present subsection we will write out a cycle responsible for the first standard relation~\eqref{guardaunpo}, this time the pull-back to $\CH^2(\cS)$ is a non-zero vertical class. 
We record for later use the following formulae:
\begin{eqnarray}
\sigma_{j,*}\rho_j^{*} c_1(\cN_{C_j/V}) & = & \pi^{*}C_j+E_j\cdot E_j, \label{sorpresa} \\ 
p_{W,*}(\xi^{N}) & = & (\pi^{*}H-E). \label{primasegre}
\end{eqnarray}
The first formula follows from the \lq\lq Key formula\rq\rq\ for $\pi^{*}C_j$, see Prop.~6.7 of~\cite{fulton}. The second formula is immediate (recall that $N=\dim\Lambda$). 
Let $j\in\{1,\ldots,n\}$. By~\Ref{hyp}{dieci} there exists an open dense $U\subset\Sigma_j$ such that, if $S\in U$, then $S\cdot E_j={\bf L}_x+ Z$, where $x\in C_j$ is the unique singular point of $\pi(S)$, ${\bf L}_x:=\pi^{-1}(x)$, and $Z$ is the residual divisor (whose support does not contain ${\bf L}_x$). It follows that
\begin{equation}
E_j\cap p_{\Lambda}^{-1}(U)=\cV_j+\cZ_j,
\end{equation}
where, for every $S\in U$, the restrictions to $E_j\cap S$ of $\cV_j$, $\cZ_j$  are equal to  ${\bf L}_x$ and  $Z$, respectively.  
We let 
\begin{equation}\label{eccoteta}
\Theta_j:=\ov{\cV}_j. 
\end{equation}
Thus  $p_{\Lambda}(\Theta_j)=\Sigma_j$, and the generic fiber of $\Theta_j\to\Sigma_j$   is a projective line.  By~\Ref{hyp}{dieci}   $\Theta_j$ is of pure codimension $2$  in $\cS$  (or empty), and hence 
\begin{equation}\label{tetavert}
\Theta_j\in\Ver^2(\cS/\Lambda). 
\end{equation}
The result below will be instrumental in  writing out the class of $\Theta_j$ in $\CH^2(\cS)$  according to Decomposition~\eqref{chowdue}.
\begin{prp}\label{prp:egizio}
Let $j\in\{1,\ldots,n\}$. Then
\begin{equation}\label{eccelente}
p_{W,*}(\Theta_j\cdot \xi^{N-1})=2E_j\cdot \pi^{*}H-E_j\cdot E_j-\pi^{*}C_j.
\end{equation}
\end{prp}
\begin{proof}
Let $\alpha,\beta\in H^0(W,\pi^{*}(H)-E)$ be generic. Then $\divisore(\alpha|_{E_j})$ and  $\divisore(\beta|_{E_j})$ are smooth divisors intersecting transversely at points $p_1,\ldots,p_s$. Let $q_i:=\pi(p_i)$ for $i\in\{1,\ldots,s\}$.  Let   $R=\PP(\la \alpha,\beta\ra)\subset\Lambda$; thus  $p_{\Lambda}^{-1}R$ represents $\xi^{N-1}$. Given $p_i$, there exists $[\lambda_i,\mu_i]\in\PP^1$ such that $\divisore(\lambda_i \alpha+ \mu_i\beta)$ contains $\pi^{-1}(q_i)$, and hence $[\lambda_i \alpha+ \mu_i\beta]\in R\cap\Sigma_j$. Conversely, every point of $R\cap\Sigma_j$ is of this type. The line $R$ intersects transversely $\Sigma_j$ because it is generic, and hence
\begin{equation}\label{chefatica}
p_{W,*}(\Theta_j\cdot \xi^{N-1})=\sigma_{j,*}\rho_j^{*}(q_1+\ldots+q_s).
\end{equation}
Thus in order to compute $p_{W,*}(\Theta_j\cdot \xi^{N-1})$ we must determine the class of the $0$-cycle $q_1+\ldots+q_s$. Let $\phi\colon C_j\times R\to C_j$ and  $\psi\colon C_j\times R\to R$ be the projections and $\cF$ the rank-$2$ vector-bundle on $C_j\times R$ defined by
$$\cF:=\phi^{*}(\cN^{\vee}_{C_j/V}\otimes\cO_{C_j}(H))\otimes\psi^{*}\cO_R(1).$$
The composition of the natural maps
\begin{equation}\label{tempietto}
\la \alpha,\beta\ra \hra H^0(W,\pi^{*}H-E)\lra H^0(E_j,\cO_{E_j}(\pi^{*}H-E))\lra  H^0(C_j,\cN^{\vee}_{C_j/V}\otimes\cO_{C_j}(H))
\end{equation}
defines a section $\tau\in H^0(\cF)$ whose zero-locus consists of points $p'_1,\ldots,p'_s$ such that $\pi(p'_i)=q_i$. Now, the zero-locus of $\tau$ represents $c_2(\cF)$, and hence
$$p_{W,*}(\Theta_j\cdot \xi^{N-1})=\sigma_{j,*}(\rho_j^{*}(\phi_{*}c_2(\cF)))$$
by~\eqref{chefatica}.  The  formula
$$c_2(\cF)=\phi^{*}(2 c_1(\cO_C(H))-c_1(\cN_{C/\PP^3}))\cdot \psi^{*}c_1(\cO_R(1)).$$
gives
\begin{equation}\label{intermedio}
p_{W,*}(\Theta_j\cdot \xi^{N-1})=2E_j\cdot \pi^{*}H-\sigma_{j,*}\left(\rho_j^{*} c_1(\cN_{C_j/ V}))\right).
\end{equation}
Then~\eqref{eccelente}  follows from the above equality together with~\eqref{sorpresa}.
\end{proof}
\begin{crl}\label{crl:egizio}
Let $j\in\{1,\ldots,n\}$. Then
\begin{equation}
\Theta_j=\xi\cdot p_W^{*}E_j+p_W^{*}(E_j\cdot \pi^{*}H-\pi^{*}C_j).
\end{equation}
\end{crl}
\begin{proof}
By~\eqref{chowdue} there exist $\beta_h\in\CH^h(W)$ for $h=0,1,2$ such that 
$$\Theta_j=\xi^2\cdot p_W^{*}\beta_0+\xi\cdot p_W^{*}\beta_1+ p_W^{*}\beta_2.$$
Restricting $p_W$ to $\Theta_j$ we get a $\PP^{N-2}$-fibration $\Theta_j\to E_j$: it follows that $\beta_0=0$ and $\beta_1=E_j$. By~\eqref{primasegre} 
\begin{equation}\label{johan}
p_{W,*}(\Theta_j\cdot\xi^{N-1})=p_{W,*}(\xi^{N}\cdot p_W^{*}E_j+\xi^{N-1}\cdot p_W^{*}\beta_2)=
(E_j\cdot \pi^{*}H-E_j\cdot E_j+\beta_2).
\end{equation}
On the other hand $p_{W,*}(\Theta_j\cdot\xi^{N-1})$ is equal to the right-hand side of~\eqref{eccelente}: equating that expression and the right-hand side of~\eqref{johan} we get 
$\beta_2=(E_j\cdot  \pi^{*}H-\pi^{*}C_j)$.
\end{proof}
\begin{crl}\label{crl:atzeco}
Let $j\in\{1,\ldots,n\}$. Then $p_W^{*}(E_j\cdot \pi^{*}H-\pi^{*}C_j)\in\Ver^2(\cS/\Lambda)$.
\end{crl}
\begin{proof}
By~\Ref{crl}{egizio} we have
\begin{equation*}
p_W^{*}(E_j\cdot \pi^{*}H-\pi^{*}C_j)=\Theta_j-\xi\cdot p_W^{*}E_j.
\end{equation*}
Now $\Theta_j\in\Ver^2(\cS/\Lambda)$ (see~\eqref{tetavert}) and $\xi\cdot p_W^{*}E_j\in\Ver^2(\cS/\Lambda)$ because it is supported on the inverse image of a hyperplane via $p_{\Lambda}$; thus   $p_W^{*}(E_j\cdot \pi^{*}H-\pi^{*}C_j)\in\Ver^2(\cS/\Lambda)$.
\end{proof}
By~\Ref{clm}{marmite} the relation  $p_W^{*}(E_j\cdot \pi^{*}H-\pi^{*}C_j)\in\Ver^2(\cS/\Lambda)$ gives a relation in $\CH(X)$ for an arbitrary smooth $X\in|\cI_C(H)|$. In fact it gives the  first standard relation~\eqref{guardaunpo}.
\subsection{Proof of the main result of the section}\label{subsec:demolink}
\begin{lmm}\label{lmm:vernoelef}
Assume that~\Ref{hyp}{ipnole} holds.
Then the projection $\CH^2(\cS)\to \CH^2(W)$ determined by~\eqref{chowdue} maps $\Ver^2(\cS/\Lambda)$ to the subspace spanned by  
 \begin{equation}\label{girini}
 (E_1\cdot \pi^{*}H-\pi^{*}C_1),\ldots,(E_j\cdot \pi^{*}H-\pi^{*}C_j),\ldots,(E_n\cdot \pi^{*}H-\pi^{*}C_n).
\end{equation}
\end{lmm}
\begin{proof}
Let $Z\subset\cS$ be an irreducible closed codimension-$2$ subset of $\cS$ such that  
\begin{equation}
\dim p_{\Lambda}(Z)<\dim Z=N.
\end{equation}
 Since the fibers of $p_{\Lambda}$ are surfaces,
\begin{equation}
\dim p_{\Lambda}(Z)=
\begin{cases}
N-2, & \text{or} \\
N-1.
\end{cases}
\end{equation}
Suppose that $\dim p_{\Lambda}(Z)=N-2$. We claim that
\begin{equation}\label{contrimm}
Z=p_{\Lambda}^{-1}(p_{\Lambda}(Z)).
\end{equation}
Since $Z\subset p_{\Lambda}^{-1}(p_{\Lambda}(Z))$, it will suffice to prove that $p_{\Lambda}^{-1}(p_{\Lambda}(Z))$ is irreducible of dimension $N$. First we notice that every irreducible component of $p_{\Lambda}^{-1}(p_{\Lambda}(Z))$ has dimension at least $N$. In fact, letting $\iota\colon p_{\Lambda}(Z)\hra \Lambda$ be the inclusion and  $\Delta_{\Lambda}\subset \Lambda\times \Lambda$ the diagonal, $p_{\Lambda}^{-1}(p_{\Lambda}(Z))$ is identified with  $(\iota,p_{\Lambda})^{-1}\Delta_{\Lambda}$, and the claim follows because $\Delta_{\Lambda}$ is a l.c.i.~of codimension $N$. 
Since every fiber of $p_{\Lambda}$ has dimension $2$, it follows that every irreducible component of  $p_{\Lambda}^{-1}(p_{\Lambda}(Z))$ dominates  $p_{\Lambda}(Z)$. 
On the other hand, since $\cod(p_{\Lambda}(Z),\Lambda)=2$, there exists an open dense $U\subset p_{\Lambda}(Z)$ such that $p_{\Lambda}^{-1}(t)$ is irreducible for all  $t\in U$  by~\Ref{hyp}{dieci}, and hence $p_{\Lambda}^{-1}(U)$ is irreducible  of dimension $N$. It follows that there is a single  irreducible component of  $p_{\Lambda}^{-1}(p_{\Lambda}(Z))$ dominating $p_{\Lambda}(Z)$, and hence $p_{\Lambda}^{-1}(p_{\Lambda}(Z))$ is irreducible (of dimension $N$).
We have proved~\eqref{contrimm}.  
Since $\Lambda$ is a projective space, $p_{\Lambda}([Z])$ is a multiple of $c_1(\cO_{\Lambda}(1))^2$. It follows that the class of $Z$ is a multiple of $\xi^2$ and hence  the projection $\CH^2(\cS)\to \CH^2(W)$  maps it to $0$. Now assume that $\dim p_{\Lambda}(Z)=N-1$. Let $Y:=p_{\Lambda}(Z)$. For $t\in\Lambda$, we let 
$S_t:=p_{\Lambda}^{-1}(t)$. 
 We distinguish between the two cases:
\begin{enumerate}
\item
 $p_{\Lambda}(Z)\notin\{\Sigma_1,\ldots,\Sigma_n\}$.
\item
There exists $j\in\{1,\ldots,n\}$ such that $p_{\Lambda}(Z)=\Sigma_j$.
\end{enumerate}
Suppose that~(1) holds. Let $Y^{sm}\subset Y$ be the subset of smooth points. If $t\in Y^{sm}$,  we may intersect the cycles $Z$ and $S_t$ in $p_{\Lambda}^{-1}(Y)$ (because $S_t$ is a l.c.i.), and the resulting cycle class $Z\cdot S_t$ belongs to $\CH^1(S_t)$.   By~\Ref{hyp}{ipnole} there exists $\Gamma\in\CH^1(W)$ such that $\Gamma|_{S_t}=Z\cdot S_t$  for $t\in Y^{sm}$.  
It follows that there exists an open dense $U\subset Y^{sm}$  such that 
\begin{equation*}
\Gamma|_{p_{\Lambda}^{-1}(U)}\equiv Z|_{p_{\Lambda}^{-1}(U)}.
\end{equation*}
(Recall that Chow groups are with $\QQ$-coefficients.)  By the localization sequence applied to $p_{\Lambda}^{-1}(U)\subset p_{\Lambda}^{-1}(Y)$, it follows that there exists a cycle $\Xi\in \CH_N(p_{\Lambda}^{-1}(Y\setminus U))$ such that  
\begin{equation*}
[Z]= \Xi+p_W^{*}(\Gamma)\cdot p_{\Lambda}^{*}([Y]).
\end{equation*}
Here, by abuse of notation, we mean cycle classes in $\CH_N(\cS)$: thus  $[Z]$ and $\Xi$ are actually the push-forwards of the corresponding classes in 
$\CH_N(p_{\Lambda}^{-1}(Y)$ and  $\CH_N(p_{\Lambda}^{-1}(Y\setminus U))$ via the obvious closed embeddings.
By~\eqref{contrimm} $\Xi$ is represented by a linear combination of varieties $p_{\Lambda}^{-1}(B_i)$, where $B_1,\ldots,B_m$ are the irreducible components of $Y\setminus U$; it follows that $\Xi=a\xi^2$ for some $a\in\QQ$. On the other hand $[Y]\in \CH^1(\Lambda)=\QQ c_1(\cO_{\Lambda}(1))$, and hence 
$p_W^{*}(\Gamma)\cdot p_{\Lambda}^{*}([Y])=b p_W^{*}(\Gamma)\xi$ for some $b\in\QQ$. 
 It follows that  the projection $\CH^2(\cS)\to \CH^2(W)$  maps $Z$ to $0$. Lastly  suppose that Item~(2) holds.  Arguing as above, one shows that there exist  
 $\Gamma\in\CH^1(W)$, an open dense $U\subset Y$, a cycle $\Xi\in \CH_N(p_{\Lambda}^{-1}(Y\setminus U))$, and $a\in\QQ$ such that  
\begin{equation*}
[Z]= \Xi+p_W^{*}(\Gamma)\cdot p_{\Lambda}^{*}([Y])+a\Theta_j.
\end{equation*}
By~\Ref{crl}{egizio}   the projection $\CH^2(\cS)\to \CH^2(W)$  maps $[Z]$ to $a(E_j\cdot \pi^{*}H-\pi^{*}C_j)$. This proves that $\Ver^2(\cS/\Lambda)$ is mapped into the subspace spanned by the elements of~\eqref{girini}. Since $[\Theta_j]$ is a vertical class and  is mapped to $(E_j\cdot \pi^{*}H-\pi^{*}C_j)$, we have proved the lemma.  
\end{proof}
{\it Proof of~\Ref{prp}{vaccini}.}
Let $P\in\QQ[x_1,\ldots,x_m]$ be homogeneous of degree $2$ and $r_1,\ldots,r_n\in\QQ$. The set of smooth $X\in|\cI_C(H)|$ such that 
\begin{equation}\label{relazione}
0=P(\zeta_1|X,\ldots,\zeta_m|X)+r_1 C_1^2+\ldots+r_n C_n^2
\end{equation}
is a countable union of closed subsets of the open dense subset of $|\cI_C(H)|$ parametrizing smooth surfaces. It follows that if the proposition is false then  there exist $P$ and  $r_1,\ldots,r_n$, not all zero, such that~\eqref{relazione} holds for all smooth $X\in |\cI_C(H)|$. Now we argue by contradiction. By~\Ref{clm}{marmite}  
\begin{equation}\label{verticale}
 p_W^{*} (P(\pi^{*}\zeta_1,\ldots,\pi^{*}\zeta_m)+\sum_{j=1}^n r_j E_j^2)\in\Ver^2(\cS/\Lambda).
\end{equation}
 By~\Ref{lmm}{vernoelef} it follows that there exist rationals $s_1,\ldots,s_n$ such that
 \begin{equation*}
P(\pi^{*}\zeta_1,\ldots,\pi^{*}\zeta_m)+\sum_{j=1}^n r_j E_j^2=\sum_{j=1}^n s_j (E_j\cdot \pi^{*}H-\pi^{*}C_j),
\end{equation*}
i.e.,
\begin{equation}\label{calimero}
0=\pi^{*}(P(\zeta_1,\ldots,\zeta_m)+\sum_{j=1}^n s_j C_j)-\sum_{j=1}^n s_j E_j\cdot \pi^{*}H+\sum_{j=1}^n r_j E_j^2.
\end{equation}
Let $\omega$ be the right hand side of~\eqref{calimero}; then the homology class of $\omega$ vanishes, and also the Abel-Jacobi image $AJ(\omega)$,  notation 
 as in~\eqref{abjac}.
 Item~(2) of~\Ref{lmm}{conti}, together with our hypothesis that there does \emph{not} exist  $\xi\in\CH^1(V)$ such that   $c_1(K_{C_j})=\xi|_{C_j}$, gives 
$r_j=0$ for $j\in\{1,\ldots,n\}$. By~\eqref{sommaga} 
\begin{equation}\label{daladier}
P(\zeta_1,\ldots,\zeta_m)+\sum_{j=1}^n s_j C_j=0,
\end{equation}
and hence $\sum_{j=1}^n s_j E_j\cdot \pi^{*}H=0$. Thus
\begin{equation}\label{xylella}
0=E_i\cdot\left(\sum_{j=1}^n s_j E_j\cdot \pi^{*}H\right)=-s_i\deg(C_i\cdot H).
\end{equation}
 for $i\in\{1,\ldots,n\}$. By hypothesis $H$ is ample, and hence $s_i=0$  follows from~\eqref{xylella}. 
Thus  $P(\zeta_1,\ldots,\zeta_m)=0$ by~\eqref{daladier}. 
\qed
\section{Noether-Lefschetz loci for linear systems of surfaces in $\PP^3$ with base-locus}\label{sec:oltrelopez}
\setcounter{equation}{0}
\subsection{The main result}
In the present section  we let $V=\PP^3$. Thus $C_1,\ldots,C_n\subset\PP^3$, and $\pi\colon W\to\PP^3$. We let 
  $\Lambda(d):= |\pi^{*}\cO_{\PP^3}(d)(-E)|$.  For $j\in\{1,\ldots,n\} $ let $\Sigma_j(d)\subset \Lambda(d)$ be the  subset $\Sigma_j$ considered in~\Ref{sec}{sezuno}; thus
 $\Sigma_j(d)$ parametrizes
 surfaces $S\in \Lambda(d)$ such that $\pi(S)$ is singular at some point of $C_j$. Let  
 $\Sigma(d):=\Sigma_1(d)\cup\ldots\cup\Sigma_n(d)$. We denote the tangent sheaf of a smooth variety $X$ by $T_X$.  Below is the main result of the present section.
\begin{thm}\label{thm:peppapig}
Suppose that 
$d\ge 5$, and that the following hold:
\begin{enumerate}
\item
$\pi^{*}\cO_{\PP^3}(d-3)(-E)$ is very ample. 
\item
$H^1(C,T_{C}(d-4))=0$.
\item
The sheaf $\cI_C$  is $(d-2)$-regular.
\item
The curves $C_1,\ldots,C_n$ are not planar.
\end{enumerate}
Then~\Ref{hyp}{ipnole} holds for $H\in |\cO_{\PP^3}(d)|$.
\end{thm}
Recall that~\Ref{hyp}{ipnole} states that~\Ref{hyp}{dieci} holds, and that Items~(1) and~(2) (our Noether-Lefschetz hypotheses) of~\Ref{hyp}{ipnole} hold.  The proof that~\Ref{hyp}{dieci} holds is elementary, and will be given in~\Ref{subsec}{vitadapecora}. 
 We will prove that  Items~(1) and~(2) of~\Ref{hyp}{ipnole} hold by 
applying Joshi's main criterion (Prop.~3.1 of~\cite{joshi}), and the main idea in Griffiths-Harris' proof of the classical Noether-Lefschetz  Theorem~\cite{grihar}, as further developed by Lopez~\cite{lopez} and Brevik-Nollet~\cite{brenol}. The proof will be outlined in~\Ref{subsec}{shaun},  details are  in the remaining subsections. 
\begin{rmk}
Choose disjoint integral smooth  curves $C_1,\ldots,C_n\subset\PP^3$ such that for each $j\in\{1,\ldots,n\}$ there does \emph{not} exist  $r\in\QQ$ such that $c_1(K_{C_j})=r c_1(\cO_{C_j}(1))$. Let   $d\gg 0$. Then the hypotheses of~\Ref{thm}{peppapig} are satisfied, and hence by~\Ref{prp}{vaccini} the following holds: 
 if $X\in|\cI_C(d)|$ is  very general, then the $0$-cycle classes $c_1(\cO_X(1))^2,  c_1(\cO_X(C_1))^2,\ldots, c_1(\cO_X(C_n))^2$ are linearly independent. Thus  the group of decomposable $0$-cycles of $X$ has rank at least $n+1$.  The proof of~\Ref{thm}{zorba} is achieved by making the above argument effective, see~\Ref{sec}{chiudo}.
\end{rmk}
\subsection{Dimension counts}\label{subsec:vitadapecora}
We will prove that, if the hypotheses  of~\Ref{thm}{peppapig} are satisfied, then~\Ref{hyp}{dieci} holds for $H\in|\cO_{\PP^3}(d)|$. First, $H$ is ample on $\PP^3$, and $\pi^{*}(H)-E$ is very ample on $W$ because it is the tensor product of the line-bundle  $\pi^{*}\cO_{\PP^3}(d-3)(-E)$, which is very ample by hypothesis, and the base-point free line-bundle  
$\pi^{*}\cO_{\PP^3}(3)$. 
Let  $\Delta(r)\subset\Lambda(r)$ be  the closed subset parametrizing singular surfaces.
\begin{prp}\label{prp:singgener}
Suppose that 
$\pi^{*}\cO_{\PP^3}(r-1)(-E)$ is very ample. Then the following hold:
\begin{enumerate}
\item
Let   $x\in C$. The linear system $|\cI_x^2(r)|\cap |\cI_C(r)|$ has base locus equal to $C$, and codimension $2$ in $|\cI_C(r)|$. If  $X$ is generic in  $|\cI_x^2(r)|\cap |\cI_C(r)|$ then it has an ODP at $x$ and no other singularity.
\item
Given $x\in W\setminus E$ there exists $S\in\Delta(r)$ which has an ODP at $x$ and is smooth away from $x$.
\item
The closed subset $\Delta(r)$  is  irreducible of  codimension $1$ in $\Lambda(r)$, and the generic $S\in\Delta(r)$ has a unique singular point, which is an ODP. 
\item
Let $j\in\{1,\ldots,n\}$. If $S$ is a generic element of  $\Sigma_j(r)$, then $\pi(S)$  has a unique singular point $x$, which is an ODP (notice that $S$ is smooth). 
\end{enumerate}
\end{prp}
\begin{proof}
Let $q\in \PP^3\setminus C$. Since $\pi^{*}\cO_{\PP^3}(r-1)(-E)$ is very ample there exists $X\in |\cI_C(r-1)|$ such that $q\notin X$. Let $P\subset\PP^3$ be a plane containing $x$ but not $q$: then $X+P\in |\cI_x^2(r)|\cap |\cI_C(r)|$ does not pass through $q$, and this proves that  $|\cI_x^2(r)|\cap |\cI_C(r)|$ has base locus equal to $C$. 
Since $\pi^{*}\cO_{\PP^3}(r-1)(-E)$ is very ample there exist $F,G\in H^0(\PP^3,\cI_C(r-1))$ and $q_1,\ldots,q_m\in (C\setminus\{x\})$ such that $V(F),V(G)$ are smooth and  transverse at each point of $C\setminus\{q_1,\ldots,q_m\}$. 
Let $P\subset\PP^3$ be a plane not passing through $x$: the pencil in $|\cI_C(r)|$ spanned by $V(F)+P$ and $V(G)+P$ does not intersect 
  $|\cI_x^2(r)|\cap |\cI_C(r)|$, and hence  $|\cI_x^2(r)|\cap |\cI_C(r)|$ has codimension at least $2$ in $|\cI_C(r)|$. The codimension is equal to $2$ because  imposing on $X\in |\cI_C(r)|$ that it  be singular at $x\in C$ is equivalent to $2$ linear equations being satisfied.
In order to show that the singularities of a generic element of  $|\cI_x^2(r)|\cap |\cI_C(r)|$ are as claimed we consider the embedding 
\begin{equation}\label{humpty}
\begin{matrix}
\PP(H^0(\PP^3,\cI_x(1))\oplus H^0(\PP^3,\cI_x(1))) & \lra & \Sigma_j(r) \\
[A,B] & \mapsto & V(A F+ B G)
\end{matrix}
\end{equation}
where $F,G$ are as above. 
The image   is a sublinear system of $|\cI_x^2(r)|\cap |\cI_C(r)|$ whose base locus is $C$, hence the generic $V(A\cdot F+ B\cdot G)$   is smooth away from $C$ by Bertini's Theorem. 
 A local computation shows that the projectivized tangent cone of $V(AF+BG)$ at $x$ is a smooth conic for generic $A,B$. Lastly let $q\in C\setminus\{x\}$. The set of $[A,B]$ such that $V(AF+BG)$   is singular at $q$ has codimension $2$ if $q\notin\{q_1,\ldots,q_m\}$, codimension $1$ if $q\in\{q_1,\ldots,q_m\}$: it follows that for generic  $[A,B]$ the surface  $V(AF+BG)$  is  smooth at all points of $C\setminus\{x\}$. This proves Item~(1). 
The remaining items are proved similarly.
\end{proof}
\begin{rmk}\label{rmk:gencone}
Let $x\in C$. The proof of~\Ref{prp}{singgener} shows that the projectivized tangent cone  at $x$ of the generic $X\in |\cI_x^2(r)|\cap |\cI_C(r)|$ is the generic conic in $\PP(T_x\PP^3)$ containing the point $\PP(T_x C)$.
\end{rmk}
\begin{prp}\label{prp:aridaje}
Suppose that $\pi^{*}\cO_{\PP^3}(r)(-E)$ is very ample and that 
$\pi^{*}\cO_{\PP^3}(r-3)(-E)$ is base point free. Then the locus of non-integral surfaces $S\in|\Lambda(r)|$ has codimension at least $4$.
\end{prp}
\begin{proof}
Let  $\Dec(r)\subset\Lambda(r)$ be the (closed) subset of  non-integral surfaces, and $\Dec(r)_1,\ldots,\Dec(r)_m$ be its irreducible components.   Let $j\in\{1,\ldots,m\}$; we will prove that   
 the locus of non-integral surfaces $S\in \Dec(r)_j$ has codimension at least $4$. Suppose first that the generic $S\in\Dec(r)_j$ contains one (at least) of the components  of $E$, say  $E_k$.
 Since $\pi^{*}\cO_{\PP^3}(r)(-E)$ is very ample, and $E_k$ is a $\PP^1$-bundle, the image of the restriction map 
 \begin{equation*}
 H^0(W,\pi^{*}\cO_{\PP^3}(r)(-E))\to H^0(E_k,\pi^{*}\cO_{\PP^3}(r)(-E)|_{E_k})
\end{equation*}
 has dimension at least $4$, and hence the locus  of $S\in 
 |\pi^{*}\cO_{\PP^3}(r)(-E)|$ which contain $E_k$ has codimension at least $4$. 
 
  Next, suppose that the generic $S\in\Dec(r)_j$ does not contain any of the  components  of $E$. Let  $\Dec(r)'_j\subset  |\cI_C(r)|$ be the image of  $\Dec(r)_j$ 
  under the natural isomorphism $\Lambda(r)\overset{\sim}{\to} |\cI_C(r)|$.
   Let  $X\in\Dec(r)'_j$  be generic; we claim that
 \begin{equation}\label{villasciarra}
 \dim(\sing X\setminus C)\ge 1. 
\end{equation}
 In fact  $X=\pi(S)$, where    $S\in\Dec(r)_j$  is generic, and since $S$ is non-integral we may write $S=S_1+S_2$ where $S_1,S_2$ are effective non-zero divisors on $W$ (we will identify effective divisors and pure codimension-$1$ subschemes of $W$ and $\PP^3$). Thus $X=X_1+X_2$, where $X_i:=\pi(S_i)$.  Since $X_1,X_2$ are effective  non-zero divisors   on $\PP^3$ (non-zero because neither $S_1$ nor $S_2$ contains a component of $E$), their intersection has dimension at least $1$. Now $X_1\cap X_2\subset\sing X$, hence in order to prove~\eqref{villasciarra} it suffices to show that  $X_1\cap X_2$ is not contained in $C$. Suppose that    $X_1\cap X_2$ is contained in $C$; then, since it has dimension at least $1$, there exists $k\in\{1,\ldots,n\}$ such that
    $X_1\cap X_2$  contains  $C_k$, and this implies that $S$ contains $E_k$, contradicting our assumption. We have proved~\eqref{villasciarra}. 
 
 Next, let  
 $p\not=q\in (\PP^3\setminus C)$, and let  $\Omega_{p,q}(r)\subset |\cI_C(r)|$ be the subset of divisors 
 $X$ which are singular at $p,q$, with degenerate quadratic terms.  
If $X\in\Dec(r)'_j$, then by~\eqref{villasciarra} there exists a couple of distinct 
 $p,q\in (X\setminus C)$ such that $X$ is singular at $p$ and $q$, with  degenerate quadratic terms (in fact the set of such couples is infinite).  Thus, if Item~(2) holds, then
 \begin{equation}
\Dec(r)'_j\subset \bigcup_{p\not=q\in (\PP^3\setminus C)}\Omega_{p,q}(r).
\end{equation}
Hence it suffices to prove that the   codimension of $\Omega_{p,q}$ in $ |\cI_C(r)|$ is  $10$ (as expected) for each $p\not=q\in (\PP^3\setminus C)$. Let $Y\in |\cI_C(r-3)|$ be a surface not containing $p$ nor $q$ (it exists because  
 $\pi^{*}\cO_{\PP^3}(r-3)(-E)$ is base point free), and consider the subset
 \begin{equation*}
P_Y:=\{Y+Z\mid  Z\in |\cO_{\PP^3}(3)|\}.
\end{equation*}
An explicit computation shows that the codimension of the set of $Z\in |\cO_{\PP^3}(3)|$  singular at $p,q$, with degenerate quadratic terms, has codimension $10$: it follows that 
$\Omega_{p,q}(r)\cap P_Y$ has codimension $10$, and hence $\Omega_{p,q}(r)$ has codimension $10$ in $|\cI_C(r)|$.
\end{proof}
\Ref{prp}{singgener} and~\Ref{prp}{aridaje} prove that, if the hypotheses  of~\Ref{thm}{peppapig} are satisfied, then~\Ref{hyp}{dieci} holds for $H\in|\cO_{\PP^3}(d)|$. 
\subsection{Outline of the proof  that the Noether-Lefschetz hypothesis holds}\label{subsec:shaun}
 Let $A$ be  an integral closed codimension-$1$ subset  of 
$\Lambda(d)$. Let  $A^{\vee}\subset\Lambda(d)^{\vee}$ be the projective dual of $A$, i.e.~the closure of the locus of projective tangent hyperplanes ${\bf T}_S A$ for $S$ a point in the smooth locus $A^{sm}$ of $A$.  Since $\pi^{*}\cO_{\PP^3}(d)(-E)$ is very ample we have the natural  embedding  $W\hra \Lambda(d)^{\vee}$, and hence it makes sense  to distinguish between the following two cases:
\begin{enumerate}
\item[(I)]
$A^{\vee}$ is not contained in $W$.
\item[(II)]
$A^{\vee}$ is  contained in $W$. 
\end{enumerate}
Thus~(I) holds if and only if, for the generic $S\in A^{sm}$,  the projective tangent hyperplane ${\bf T}_S A$ is a base point free linear subsystem of $\Lambda(d)$. On the other hand, examples of codimension-$1$ subsets of $\Lambda(d)$ for which (II) holds are given by  $\Delta(d)$ and by $\Sigma_j(d)$ for $j\in\{1,\ldots,n\}$. In fact
$\Delta(d)^{\vee}=W$ and $\Sigma_j(d)^{\vee}=E_j$. 
The last equality holds because   $S\in\Lambda(d)$ belongs to  $\Sigma_j(d)$ if and only if it is tangent to $E_j$, thus  $\Sigma_j(d)=E_j^{\vee}$, and hence  $\Sigma_j(d)^{\vee}=E_j$ 
 by projective duality.
Let $\NL(\Lambda(d)\setminus\Delta(d))$ be the  \emph{Noether-Lefschetz locus}, i.e.~the set of those smooth surfaces $S\in\Lambda(d)$ such that the restriction map $\Pic(W)_{\QQ}\to \Pic(S)_{\QQ}$ is \emph{not} surjective. As is well-known   $\NL(\Lambda(d)\setminus\Delta(d))$  is a countable union of closed subsets of $\Lambda(d)\setminus\Delta(d)$.  
In~\Ref{subsec}{primocaso}  we will apply Joshi's criterion (Proposition~3.1 of~\cite{joshi}) in  order to prove the following result. 
\begin{prp}\label{prp:pazienza}
Suppose that $d\ge 5$ and that the following hold:
\begin{enumerate}
\item
$\pi^{*}\cO_{\PP^3}(d)(-E)$ is ample. 
\item
$H^1(C,T_{C}(d-4))=0$.
\item
The sheaf $\cI_C$ (on $\PP^3$) is $(d-2)$-regular.
\end{enumerate}
Let $A\subset\Lambda(d)$ be an integral closed subset of codimension $1$, and suppose that there exists  $S\in (A\setminus\Delta(d))$ such that $A$ is smooth at $S$,  and the projective tangent space ${\bf T}_S A$ is a base-point free linear subsystem of $\Lambda(d)$. Then $A\setminus\Delta(d)$ does not belong to the Noether-Lefschetz locus $NL(\Lambda(d)\setminus\Delta(d))$.  
\end{prp}
The above result deals with codimension-$1$ subsets $A\subset\Lambda(d)$ for which~(I) above holds.  
Thus, in order to finish the proof of~\Ref{thm}{peppapig}, it remains to deal with those $A$ such that~(II) above 
holds.  
\begin{dfn}\label{dfn:lamgam}
Given $p\in W$ and $F\subset T_p W$  a vector subspace, we let 
\begin{equation}\label{lampif}
\Lambda_{p,F}(d):=\{S\in |\cI_p\otimes \pi^{*}\cO_{\PP^3}(d)(-E)| : F\subset T_p S\}.
\end{equation}
Let $\Gamma(d):=|\cI_C(d)|$. We have a tautological  identification
$\Lambda(d)\overset{\sim}{\lra}\Gamma(d)$: we let $\Gamma_{p,F}(d)\subset\Gamma(d)$ be the image of $\Lambda_{p,F}(d)$, and for $j\in\{1,\ldots,n\}$ we let 
$\Pi_j(d)\subset\Gamma(d)$ be the image of $\Sigma_j(d)$. 
\end{dfn}
Notice that $\Lambda_{p,F}(d)$ and $\Gamma_{p,F}(d)$ are  linear subsystems of $\Lambda(d)$ and  $\Gamma(d)$ respectively. 
In~\Ref{subsec}{sgocciola} we will prove the  result below by applying  an idea of Griffiths-Harris~\cite{grihar} as further developed by Lopez~\cite{lopez} and Brevik-Nollet~\cite{brenol}. 
\begin{prp}\label{prp:dirtan}
 Suppose that the following hold:
\begin{enumerate}
\item
$d\ge 4$ and $\pi^{*}\cO_{\PP^3}(d-3)(-E)$ is very ample.
\item
None of the curves $C_1,\ldots,C_n$ is planar.
\end{enumerate}
Let $X$ be a very general element  
\begin{enumerate}
\item[(a)]
of $\Gamma_{p,F}(d)$, where either $p\notin E$, or else $p\in E$ and  
\begin{equation}\label{contenimento}
T_p(\pi^{-1}(\pi(p)))\not\subset  F\subsetneq T_pE,
\end{equation}
\item[(b)]
or of $\Pi_j(d)$ for some  $j\in\{1,\ldots,n\}$. 
\end{enumerate}
Then the Chow group  $\CH^1(X)_{\QQ}$ is generated by $c_1(\cO_X(1))$ and the classes of $C_1,\ldots,C_n$. 
\end{prp}
Granting~\Ref{prp}{dirtan}, let us prove that the statement of~\Ref{thm}{peppapig} holds for $A$ such that $A^{\vee}$ is  contained in $W$.  We will distinguish between the following two cases:
\begin{enumerate}
\item[(IIa)]
$A\not\in\{\Sigma_1(d),\ldots,\Sigma_n(d)\}$.
\item[(IIb)]
$A\in\{\Sigma_1(d),\ldots,\Sigma_n(d)\}$.
\end{enumerate}
Suppose that~(IIa) holds. By projective duality $A$ is the closure of
\begin{equation}\label{rigata}
\bigcup_{p\in (A^{\vee})^{sm} }\Lambda_{p,T_p A^{\vee}}
\end{equation}
Let  $p \in (A^{\vee})^{sm}$ be generic. We claim that Item~(a) of~\Ref{prp}{dirtan} hold for $p$ and $F=T_p A^{\vee}$. In fact if  $A^{\vee}\not\subset E$ then $p\notin E$ by genericity. If $A^{\vee}\subset E$ then $A^{\vee}$  is contained in $E_j$ for a certain $j\in\{1,\ldots,n\}$. We claim that $A^{\vee}$ is a proper subset of  $E_j$, and it is not equal to a fiber of the restriction of $\pi$ to $E_j$. In fact, if $A^{\vee}=E_j$, then $A=E_j^{\vee}=\Sigma_j(d)$, and that contradicts the assumption that(IIa) holds. Now suppose that $A^{\vee}=\pi^{-1}(q)$ for a certain $q\in C_j$. Let   $S\in A$ be generic.  Since $A$ is the closure of~\eqref{rigata},  $S$ is tangent to $\pi^{-1}(q)$, and hence contains $\pi^{-1}(q)$ because $S$ has degree $1$ on every fiber of $E_j\to C_j$.  It follows that $S$  is tangent to $E_j$, and hence $A\subset E_j^{\vee}=\Sigma_j(d)$, contradicting the hypothesis that (IIa) holds. 

Thus Item~(a) of~\Ref{prp}{dirtan} hold for $p \in (A^{\vee})^{sm}$ generic and $F=T_p A^{\vee}$, and hence if $S\in \Lambda_{p,T_p A^{\vee}}(d)$ is  very general, then $\CH^1(X)_{\QQ}$ is generated by $c_1(\cO_X(1))$ and the classes of $C_1,\ldots,C_n$. 

On the other hand, 
since $A\not\subset\Sigma(d)$,  $S$ intersects transversely  $E$, and hence the restriction of $\pi$ to $S$ is an isomorphism $S\overset{\sim}{\lra} X$. 
It follows that $\CH^1(S)_{\QQ}$ is equal to the image of $\CH^1(W)_{\QQ}\to\CH^1(S)_{\QQ}$. This proves that there exists $S\in A$ such that $\CH^1(S)_{\QQ}$ is equal to the image of $\CH^1(W)_{\QQ}\to\CH^1(S)_{\QQ}$. Actually our argument proves that there exists such an $S$ which is smooth if $A\not=\Delta(d)$, and that if $A=\Delta(d)$ there exists such an $S$ whose singular set consists of a single ODP. If the former holds, then we are done because $\NL(A\setminus\Delta(d))$ is a countable union of closed subsets of $A\setminus\Delta(d)$, and we have shown that the complement is non-empty. If the latter holds, let $\Delta(d)^0\subset\Delta(d)$ be the  open dense subset parametrizing surfaces with an ODP and no other singular point,   
then   the set of $S\in \Delta(d)^0$ such that $\CH^1(W)\to \CH^1(S)$ is \emph{not} surjective is a countable union of closed subsets of $\Delta(d)^0$ (take a simultaneous resolution), and we are done because we have shown that the complement is non empty.

Lastly suppose that~(IIb) holds, i.e.~$A=\Sigma_j(d)$ for a certain $j\in\{1,\ldots,n\}$. By~\Ref{prp}{singgener} there exists an open dense subset $\Sigma_j(d)^0\subset \Sigma_j(d)$ with the following property.  
If $S\in \Sigma_j(d)^0$ and $X=\pi(S)$, then $X$ has a unique singular point, call it $x$ (obviously $x\in C_j$),   which is an ordinary node, and the  restriction of $\pi$ to $S$ is the blow-up of  $X$ with center $x$ (in particular $S$ is smooth). Now suppose that $S\in \Sigma_j(d)^0$ is  very general. Then  by~\Ref{prp}{dirtan} the Chow group $\CH^1(S)_{\QQ}$ is generated by  the image of $\CH^1(W)_{\QQ}\to \CH^1(S)_{\QQ}$ and the class of $\pi^{-1}(x)$. Now notice that  the set of $S\in\Sigma_j(d)^0$ such that $\CH^1(S)$ is \emph{not} generated by the image of $\CH^1(W)_{\QQ}$ together with $\pi^{-1}(x)$ is a countable union of closed subsets of $\Sigma_j(d)^0$; since the complement is not empty,  we are done.
\qed

Summing up: we have shown that in order to prove~\Ref{thm}{peppapig} it suffices to prove~ \Ref{prp}{pazienza}  and~\Ref{prp}{dirtan}. The proofs are in the following subsections.
\subsection{Infinitesimal Noether-Lefschetz results}
We will recall a key result of K.~Joshi. Let $U\subset H^0(W,\pi^{*}\cO_{\PP^3}(d)(-E))$ be a subspace and  $\sigma\in U$ be  non-zero. We let  $S:=V(\sigma)$, and we assume that $S$ is smooth.  Let  $\gm_{\sigma,U}\subset \cO_{{\PP(U)},[\sigma]}$ be the maximal ideal and $\cT_{\sigma,U}:=\Spec(\cO_{{\PP(U)},[\sigma]}/\gm^2_{\sigma})$ be  the first-order infinitesimal neighborhood of $[\sigma]$ in $\PP(U)$. We let $\cS_{\sigma,U}\to 
\cT_{_{\sigma,U}}$ be the   restriction of the family $\cS_{\Lambda}\to\Lambda$ to $\cT_{\sigma,U}$.  The Infinitesimal Noether Lefschetz (INL) 
 Theorem is valid at $(U,\sigma)$ (see Section 2 of~\cite{joshi}) if the group of line-bundles on $\cS_{\sigma,U}$ is equal to the image of the composition
 \begin{equation}
\Pic(W)\lra \Pic(W\times_{\CC}{\cT_{\sigma,U}})\lra \Pic(\cS_{\sigma,U}).
\end{equation}
Let $A\subset\Lambda(d)$ be an integral closed subset. Let  $[\sigma]$ be a smooth point of $A$,  and suppose that $S=V(\sigma)$ is smooth. Let  $\PP(U)$ be the projective tangent space to $A$ at $[\sigma]$.  If the INL Theorem holds for  $(U,\sigma)$ then $A\setminus\Delta(d)$ does \emph{not} belong to the Noether-Lefschetz locus  $NL(\Lambda(d)\setminus\Delta(d))$. 

Joshi~\cite{joshi} gave a  cohomological condition which suffices for the validity of the INL Theorem.   Suppose that $U\subset H^0(W,\pi^{*}\cO_{\PP^3}(d)(-E))$  generates 
$\pi^{*}\cO_{\PP^3}(d)(-E)$; we let $M(U)$ be the locally-free sheaf on $W$ fitting into the exact sequence
\begin{equation}\label{eccomu}
0\lra M(U)\lra U\otimes\cO_W\lra \pi^{*}\cO_{\PP^3}(d)(-E)\lra 0.
\end{equation}
\begin{prp}[K.~Joshi, Prop.~3.1 of~\cite{joshi}]\label{prp:kirti}
Let $U\subset H^0(W,\pi^{*}\cO_{\PP^3}(d)(-E))$ be a subspace which generates 
$\pi^{*}\cO_{\PP^3}(d)(-E)$. 
Let $0\not=\sigma\in U$.  Suppose that 
 $S=V(\sigma)$ is smooth, and that  
\begin{enumerate}
\item[(a)]
$H^1(W,\Omega^2_W\otimes\pi^{*}\cO_{\PP^3}(d)(-E))=0$. 
\item[(b)]
$H^1(W,M(U)\otimes K_W\otimes\pi^{*}\cO_{\PP^3}(d)(-E))=0$.
\end{enumerate}
Then the INL Theorem holds at $(U,\sigma)$.
\end{prp}
\subsection{The generic tangent space is a base-point free linear system}\label{subsec:primocaso}
We will prove~\Ref{prp}{pazienza} by applying~\Ref{prp}{kirti}. 
\begin{lmm}\label{lmm:pazienza}
Suppose that
\begin{equation}\label{duezeri}
0=H^1(\PP^3,\cI_C\otimes T_{\PP^3}(d-4))=H^1(C,T_C(d-4)).
\end{equation}
Then $H^1(W,\Omega^2_W\otimes\pi^{*}\cO_{\PP^3}(d)(-E))=0$.
\end{lmm}
\begin{proof}
Since $\Omega^2_W\cong T_W\otimes K_W$ it is equivalent to prove that 
\begin{equation}\label{dadimo}
0=H^1(W, T_W\otimes K_W\otimes\pi^{*}\cO_{\PP^3}(d)(-E))=H^1(W, T_W\otimes \pi^{*}\cO_{\PP^3}(d-4)).
\end{equation}
 Let $\rho\colon E\to C$ be the restriction of $\pi$. Restricting the differential of $\pi$ to $E$, one gets an exact sequence 
\begin{equation}\label{aniene}
0\lra  \cO_W(E)|_{E}\lra \rho^{*} \cN_{C/\PP^3}\lra \xi\lra 0
\end{equation}
of sheaves on $E$, where $\xi$ is an invertible sheaf.  
Let $\iota\colon E\hra W$ be the inclusion map. The differential of $\pi$ gives the exact sequence
\begin{equation}
0\lra  T_W\otimes \pi^{*}\cO_{\PP^3}(d-4)\lra \pi^{*}T_{\PP^3}(d-4)\lra \iota_{*}(\xi \otimes \rho^{*}\cO_{C}(d-4))\lra 0.
\end{equation}
Below is a piece of the associated long exact sequence of cohomology:
\begin{equation}\label{raritan}
\scriptstyle
H^0(W,\pi^{*}T_{\PP^3}(d-4))\to H^0(W,\iota_{*}(\xi\otimes \rho^{*}\cO_{C}(d-4))) \to   H^1(W,T_W\otimes \pi^{*}\cO_{\PP^3}(d-4)) 
\to H^1(W,\pi^{*}T_{\PP^3}(d-4)).
\end{equation}
We claim that $H^1(W, \pi^{*}T_{\PP^3}(d-4)) =0$. In fact the spectral sequence associated to $\pi$ and abutting to the cohomology $H^q(W, \pi^{*}T_{\PP^3}(d-4))$ gives an exact sequence 
\begin{equation}
0\to H^1(\PP^3,\pi_{*}\pi^{*}T_{\PP^3}(d-4)) \to H^1(W, \pi^{*}T_{\PP^3}(d-4))\to H^0(\PP^3,R^1\pi_{*}\pi^{*}T_{\PP^3}(d-4))\to 0.
\end{equation}
Now $\pi_{*}\pi^{*}T_{\PP^3}(d-4)\cong T_{\PP^3}(d-4)$ and hence $H^1(\PP^3,\pi_{*}\pi^{*}T_{\PP^3}(d-4))=0$. Moreover  
$R^1\pi_{*}\pi^{*}T_{\PP^3}(d-4)=0$ because $R^1\pi_{*}\cO_W=0$, and hence $H^1(W, \pi^{*}T_{\PP^3}(d-4)) =0$. By~\eqref{raritan},  in order to complete the proof it suffices to show that the map
\begin{equation}\label{delaware}
H^0(W,\pi^{*}T_{\PP^3}(d-4))\to H^0(W,\iota_{*}(\xi\otimes \rho^{*}\cO_{C}(d-4))) 
\end{equation}
is surjective. The long exact cohomology sequence associated to~\eqref{aniene} gives an isomorphism 
\begin{equation*}
H^0(C,\cN_{C/\PP^3}(d-4))\overset{\sim}{\lra} H^0(W,\iota_{*}(\xi\otimes \rho^{*}\cO_{C}(d-4))),
\end{equation*}
 and moreover the map of~\eqref{delaware} is identified with the composition
\begin{equation}
H^0(\PP^3,T_{\PP^3}(d-4))\overset{\alpha}{\lra} H^0(C, T_{\PP^3}(d-4)|_C) \overset{\beta}{\lra} H^0(C,\cN_{C/\PP^3}(d-4)).
\end{equation}
The map $\alpha$ is surjective by the first vanishing in~\eqref{duezeri}, while $\beta$ is surjective  by the second vanishing in~\eqref{duezeri}. 
\end{proof}
Let  $U\subset H^0(\PP^3,\cI_{C}(d))$ be a subspace which generates $\cI_{C}(d)$; we let $\ov{M}(U)$ be the sheaf on $\PP^3$ fitting into the exact sequence
\begin{equation}\label{fascionuc}
0\lra \ov{M}(U)\lra U\otimes\cO_{\PP^3}\lra \cI_{C}(d)\lra 0.
\end{equation}
The curve $C$ is a local complete intersection because $C$ is smooth, and hence $\ov{M}(U)$ is locally-free. 
\begin{lmm}\label{lmm:santa}
Suppose that the hypotheses of~\Ref{lmm}{pazienza} hold and that in addition the sheaf $\cI_{C}$ is $d$-regular. Let $U\subset H^0(\PP^3,\cI_{C}(d))$ be a subspace which generates $\cI_{C}(d)$, and let $c$ be its codimension. Then $\bigwedge^p \ov{M}(U)$ is $(p+c)$-regular.  
\end{lmm}
\begin{proof}
Let $\ov{M}:=\ov{M}(H^0( \cI_{C}(d)))$. Then $\ov{M}$ is $1$-regular: in fact $H^1(\PP^3,\ov{M})=0$ because the exact sequence induced by~\eqref{fascionuc} on $H^0$ is exact by definition, and  $H^i(\PP^3,\ov{M}(1-i))=0$ for $i\ge 2$ because $\cI_{C}$ is $d$-regular. It follows that 
 $\bigwedge^p\ov{M}$ is $p$-regular (Corollary 1.8.10 of~\cite{roblaz}). Then, arguing 
as in  the proof of the Lemma on p.~371 of~\cite{kim} (see also Example 1.8.15 of~\cite{roblaz}) one gets that  $\bigwedge^p \ov{M}(U)$ is 
$(p+c)$-regular.
\end{proof}

\medskip
\noindent
{\it Proof of~\Ref{prp}{pazienza}.} 
By hypothesis there exists a smooth point $[\sigma]$ of $(A\setminus\Delta(d))$, such that   the projective tangent space ${\bf T}_S A$ is a base-point free codimension-$1$ linear subsystem of $\Lambda$.  We have  ${\bf T}_S A=\PP(U)$, where 
$U\subset H^0(\pi^{*}\cO_{\PP^3}(d)(-E))$ is a codimension-$1$ subspace which generates $\cO_{\PP^3}(d)(-E)$. We will prove that  the INL Theorem holds for  $(U,\sigma)$, and~\Ref{prp}{pazienza} will follow. 
By Joshi's \Ref{prp}{kirti}  it suffices to prove that the following hold:
\begin{enumerate}
\item[(a)]
$H^1(W,\Omega^2_W\otimes\pi^{*}\cO_{\PP^3}(d)(-E))=0$. 
\item[(b)]
$H^1(W,M(U)\otimes K_W\otimes\pi^{*}\cO_{\PP^3}(d)(-E))=0$.
\end{enumerate}
We start by noting that, since $T_{\PP^3}$ is $-1$-regular, and by hypothesis $\cI_C$ is $(d-2)$ regular, the sheaf $\cI_C\otimes T_{\PP^3}$ is $(d-3)$-regular, see Proposition 1.8.9 and Remark 1.8.11 of~\cite{roblaz}. Thus the hypotheses of~\Ref{lmm}{pazienza} are satisfied, and hence
Item~(a) holds. Let us prove that Item~(b) holds. Tensoring~\eqref{eccomu} by 
$K_W\otimes\pi^{*}\cO_{\PP^3}(d)(-E)\cong \pi^{*}\cO_{\PP^3}(d-4)$ and taking cohomology we get an  exact sequence
\begin{multline}
\scriptstyle
0\to H^0(W,M(U)\otimes \pi^{*}\cO_{\PP^3}(d-4))\to 
U\otimes H^0(W,\pi^{*}\cO_{\PP^3}(d-4))\overset{\alpha}{\to} H^0(W,\pi^{*}\cO_{\PP^3}(2d-4)(-E))\to \\
\scriptstyle
\to H^1(W,M(U)\otimes \pi^{*}\cO_{\PP^3}(d-4))\to U\otimes H^1(W,K_W\otimes\pi^{*}\cO_{\PP^3}(d)(-E)).
\end{multline}
Now $H^1(W,K_W\otimes\pi^{*}\cO_{\PP^3}(d)(-E))=0$ because by hypothesis $\pi^{*}\cO_{\PP^3}(d)(-E)$ is ample. Thus it suffices to prove that the map 
$\alpha$ is surjective.
We have an identification $H^0(W,\pi^{*}\cO_{\PP^3}(d)(-E))=H^0(\PP^3,\cI_C(d))$, and hence $U$ is identified with a codimension-$1$ subspace of 
$H^0(\PP^3,\cI_C(d))$ that we will denote by the same symbol.  Clearly it suffices to prove that the natural map
\begin{equation}\label{pezzo}
U\otimes H^0(\PP^3,\cO_{\PP^3}(d-4))\lra H^0(\PP^3,\cI_C(2d-4))
\end{equation}
is surjective. Tensorize Exact Sequence~\eqref{fascionuc} by $\cO_{\PP^3}(d-4)$ and take the associated long exact sequence of cohomology: then~\eqref{pezzo} appears in that exact sequence, and hence it suffices to prove that $H^1(\PP^3,\ov{M}(U)(d-4))=0$. By~\Ref{lmm}{santa} the sheaf  $\ov{M}(U)$ is $2$-regular, and by hypothesis $d\ge 5$:  the required vanishing follows.
\qed
\subsection{All tangent spaces at smooth points are linear systems with a base-point}\label{subsec:sgocciola}
We will prove~\Ref{prp}{dirtan}.   
We start with an elementary result. 
\begin{lmm}\label{lmm:genpif}
Assume that $\pi^{*}\cO_{\PP^3}(r-3)(-E)$ is very ample.  Let $p\in W$ and $F\subset T_pW$ be a subspace
such that one of the following holds:
\begin{enumerate}
\item
$p\notin E$ and $F\not= T_pW$,
\item
$p\notin E$ and $F= T_pW$,
\item
 $p\in E$,  and $T_p(\pi^{-1}(\pi(p)))\not\subset F\subsetneq T_pE$.
\end{enumerate}
Let $X\in\Gamma_{p,F}(r)$ (see~\Ref{dfn}{lamgam}) be  generic. Then $X$ is smooth if Item~(1) or~(3) holds, while $X$ has an ODP at $q=\pi(p)$ and is smooth elsewhere if Item~(2) holds.
\end{lmm}
\begin{proof}
Suppose first that~(1) or~(2) holds, i.e.~$p\notin E$, and let $q:=\pi(p)$. The linear system $\Gamma_{p,F}(r)$ has base locus $C\cup\{q\}$. In fact, let $z\in(\PP^3\setminus C\setminus\{q\})$; then there exists  $Y\in |\cI_C(r-2)|$   not containing $z$ because  $\pi^{*}\cO_{\PP^3}(r-2)(-E)$ is very ample, and a quadric $Q\in\Gamma_{p,F}(2)$ not containing $z$. Thus $Y+Q$ is an element of  $\Gamma_{p,F}(r)$ which does not contain $z$. Hence the generic 
 $X\in\Gamma_{p,F}(r)$ is smooth away from $C\cup\{q\}$ by Bertini. Considering  $Y+Q\in\Gamma_{p,F}(r)$ as above we also get that the beahviour in $q$ of the generic element of $\Gamma_{p,F}(r)$ is as claimed. It remains to prove that the generic $X\in\Gamma_{p,F}(r)$ is smooth at every point of $C$, i.e.~that $\Gamma_{p,F}(r)$ is not a subset of 
 $\Sigma(r)$. 
 The  proof that  $\Gamma_{p,F}(r)$ has base locus $C\cup\{q\}$ proves also that
 \begin{equation}\label{menoerreuno}
\dim \Gamma_{p,F}(r)=\dim|\cI_C(r)|-\dim F-1. 
\end{equation}
 Thus in order to prove that  $\Gamma_{p,F}(r)$ is not a subset of 
 $\Sigma(r)$, it suffices  to prove that for $x\in C$
\begin{equation}\label{ottavilla}
\dim |\cI_x^2(r)|\cap \Gamma_{p,F}(r) \le \dim|\cI_C(r)|-\dim F-3.
\end{equation}
By Item~(1) of~\Ref{prp}{singgener}, $\dim  |\cI_x^2(r)|\cap |\cI_C(r)|=\dim|\cI_C(r)|-2$, and hence~\eqref{ottavilla} is equivalent to   
\begin{equation}\label{arcobaleno}
\cod( |\cI_x^2(r)|\cap \Gamma_{p,F}(r), |\cI_x^2(r)|\cap |\cI_C(r)|)=\dim F+1.
\end{equation}
We must check that  imposing to $X\in |\cI_x^2(r)|\cap |\cI_C(r)|$ that it contains $q$ and that $d\pi(p)(F)\subset T_qX$, gives  $\dim F+1$ linearly independent conditions.  By Item~(1) 
of~\Ref{prp}{singgener}, there exists   $Y\in |\cI_x^2(r-2)|\cap |\cI_C(r-2)|$   not 
containing $q$. Consider the  linear subsystem  $A\subset |\cI_x^2(r)|\cap |\cI_C(r)|$  whose elements are $Y+Q$, where $Q\in |\cO_{\PP^3}(2)|$;   imposing to $X\in A$ that it contains $q$ and that $d\pi(p)(F)\subset T_qX$, gives  $\dim F+1$ linearly independent conditions,  and hence~\eqref{arcobaleno} follows. This finishes the proof that if~(1)  holds then ($1'$) holds, and that  if~(2)  holds then ($2'$) holds. 

Now suppose that~(3) holds. Suppose that $F=\{0\}$, and let  $S\in \Lambda_{p,F}(r)= |\cI_p\otimes\pi^{*}\cO_{\PP^3}(r)(-E)|$
 be generic. Then $S$ is  smooth at $p$ because $\pi^{*}\cO_{\PP^3}(r)(-E)$ is very ample, and by Bertini's Theorem it is smooth away from $p$ as well. In order to prove that $X=\pi(S)$ is smooth we must check that $S$ does not contain any of the lines ${\bf L}_x:=\pi^{-1}(x)$ for $x\in C$. Since  $\pi^{*}\cO_{\PP^3}(r)(-E)$ is very ample,
\begin{equation}\label{bartolo}
\scriptstyle
\cod(|\cI_{{\bf L}_x}\otimes\pi^{*}\cO_{\PP^3}(r)(-E)| \cap |\cI_p\otimes\pi^{*}\cO_{\PP^3}(r)(-E)|,\,|\cI_p\otimes\pi^{*}\cO_{\PP^3}(r)(-E)|)=
\begin{cases}
\scriptstyle 1 & \scriptstyle \text{if $x=q$},\\
\scriptstyle 2 & \scriptstyle \text{if $x\not=q$}.
\end{cases}
\end{equation}
It follows  that a generic  $S\in |\cI_p\otimes\pi^{*}\cO_{\PP^3}(r)(-E)|$ does not contain any  ${\bf L}_x$. 

We are left to deal with the case of a $1$-dimensional  $F\subset T_pE$ transverse to $T_p(\pi^{-1}(q))$. Let $Z\subset W$ be the $0$-dimensional scheme of length $2$ supported at $p$, with tangent space $F$; thus $Z\subset E$. We must prove that a generic  $S\in |\cI_Z\otimes\pi^{*}\cO_{\PP^3}(r)(-E)|$ is  smooth and contains no line ${\bf L}_x$ where $x\in C$. 

We claim that the (reduced) base-locus of $|\cI_Z\otimes\pi^{*}\cO_{\PP^3}(r)(-E)|$ is $p$. In fact no  $z\in({\bf L}_q\setminus\{p\})$ is in the base-locus of   
 $|\cI_Z\otimes\pi^{*}\cO_{\PP^3}(r)(-E)|$ because ${\bf L}_q$ is a line and there exists  $S\in |\cI_Z\otimes\pi^{*}\cO_{\PP^3}(r)(-E)|$ which is not tangent to ${\bf L}_q$ at $p$. Moreover no $z\in(W\setminus {\bf L}_q)$  is in the base-locus of   
 $|\cI_Z\otimes\pi^{*}\cO_{\PP^3}(r)(-E)|$ because of the following argument. There exist  $T\in |\cI_p\otimes\pi^{*}\cO_{\PP^3}(r-1)(-E)|$  not containing $z$, and a plane $P\subset\PP^3$  containing $q$ and  not containing $\pi(z)$; then $(T+P)\in |\cI_Z\otimes\pi^{*}\cO_{\PP^3}(r)(-E)|$  does not contain $z$.  This proves  that the (reduced) base-locus of $|\cI_Z\otimes\pi^{*}\cO_{\PP^3}(r)(-E)|$ is $p$; it follows that the generic 
 $S\in |\cI_Z\otimes\pi^{*}\cO_{\PP^3}(r)(-E)|$ is  smooth. 
 
 We finish by showing that~\eqref{bartolo} holds with  $\cI_p$ replaced by $\cI_Z$. 
The case $x=q$ is immediate. If $x\in C\setminus\{q\}$ we get the result by considering elements 
$(T+P)\in |\cI_Z\otimes\pi^{*}\cO_{\PP^3}(r)(-E)|$ where $P$ is a fixed plane containing $q$ and not containing $x$, and  
$T\in |\cI_p\otimes\pi^{*}\cO_{\PP^3}(r-1)(-E)|$.
\end{proof}
\begin{rmk}\label{rmk:verano}
The proof of~\Ref{lmm}{genpif} shows  that, if Item~(2) holds,  the projectivized tangent cone at $q$ of the generic $X\in\Gamma_{p,F}(r)$ is the generic conic in $\PP(T_q\PP^3)$.
\end{rmk}
\medskip 
\n
{\it Proof of~\Ref{prp}{dirtan}.}
 Let $r\in\{d-1,d\}$. Suppose that $p\in W$, $F\subset T_pW$, and either $p\notin E$, or else $p\in E$ and~\eqref{contenimento} holds. By~\Ref{lmm}{genpif}  there exists an open dense subset $\cU_{p,F}(r)\subset \Gamma_{p,F}(r)$ such that for $X\in \cU_{p,F}(r)$ the following holds:
\begin{enumerate}
\item
$X$ is smooth if $p\notin E$ and $F\not= T_p W$, or $p\in E$.
\item
$X$ has an ODP at $q=\pi(p)$, and  is smooth elsewhere, if  $p\notin E$ and $F= T_p W$.
\end{enumerate}
Similary, let  $j\in\{1,\ldots,n\}$, and $q\in C_j$. By~\Ref{prp}{singgener} there exists an open dense subset $\cU_{q,j}(r)\subset |\cI_q^2(r)|\cap \Sigma_j(r)$ such that every $X\in \cU_{q,j}(r)$ has an ODP at $q$ and is smooth elsewhere. 
It will suffice to prove that if $X$ is a very general element of $\cU_{p,F}(r)$ or of $\cU_{q,j}(r)$, then 
$\CH^1(X)_{\QQ}$ is generated by $c_1(\cO_X(1))$ and the classes of $C_1,\ldots,C_n$. 
Notice that if $X$ is an element of $\cU_{p,F}(r)$ or of $\cU_{q,j}(r)$, then 
 $X$ is $\QQ$-factorial. More precisely: if $D$ is a Weil divisor on $X$ then $2D$ is a Cartier divisor. Let $\NL(\cU_{p,F}(d))\subset \cU_{p,F}(d)$ be the subset of $X$ such that $\Pic(X)\otimes\QQ$ is not generated by $\cO_X(1)$ and $\cO_X(2C_1),\ldots,\cO_X(2C_n)$, and define similarly 
  $\NL(\cU_{q,j}(d))\subset \cU_{q,j}(d)$. Then  $\NL(\cU_{p,F}(d))$ is a countable union of closed subsets of $\cU_{p,F}(d)$ (there exists a simultaneous resolution if the surfaces in  $\cU_{p,F}(d)$ are not smooth), and similarly for   $\NL(\cU_{q,j}(d))$. 
  Hence it suffices to prove that    $\cU_{p,F}(d)\setminus\NL(\cU_{p,F})(d)$ and  $\cU_{q,j}(d)\setminus \NL(\cU_{q,j}(d))$ are not empty. 

 Let $Y$ be an element of $\cU_{p,F}(d-1)$ or of $\cU_{q,j}(d-1)$, and let $X$ be a generic element of $\cU_{p,F}(d)$ or of $\cU_{q,j}(d)$. 
 Since $\pi^{*}\cO_{\PP^3}(d)(-E)$ is very ample and $X$ is generic,  the intersection of $X$ and $Y$  is reduced, and there exists an integral curve $C_0\subset\PP^3$ such that its irreducible decomposition is
\begin{equation}\label{decomposizione}
X\cap Y=C_0\cup C_1\cup\ldots\cup C_n.
\end{equation}
Now let $P\subset\PP^3$ be a generic plane, in particular transverse to $C_0\cup C_1\cup\ldots\cup C_n$. Let $X=V(f)$, $Y=V(g)$ and $P=V(l)$. Let
\begin{equation}
\cZ:=V(g\cdot l+t f)\subset \PP^3\times\aff^1.
\end{equation}
The projection $\cZ\to\aff^1$ is a family of degree-$d$ surfaces, with central fiber $Y+P$. The $3$-fold $\cZ$ is singular. First $\cZ$ is singular at the points $(x,0)$ such that $x\in X\cap Y\cap P$, and it has an ODP at each of these points  because $P$ is transverse to $C_0\cup C_1\cup\ldots\cup C_n$. Moreover   
\begin{enumerate}
\item[(I)]
$\cZ$  has no other singularities if we are dealing with $\cU_{p,F}(d)$ and $F\not= T_pW$,  
\item[(II)]
$\cZ$ is also singular at $\{q\}\times\aff^1$ if  we are dealing with $\cU_{p,F}(d)$ and $F=T_p W$, or if we are dealing with $\cU_{q,j}(d)$. 
\end{enumerate}
We desingularize $\cZ$ as follows. The ODP's are eliminated by a small resolution (we follow p.~35 of~\cite{grihar}, and choose a specific small resolution among  the many possible ones), while to desingularize  $\{q\}\times\aff^1$ we blow-up that curve:  let $\wh{\cZ}\to\cZ$ be the birational morphism. Then $\wh{\cZ}$ is smooth (if  $p\notin E$ and $F=T_p W$, or if we are dealing with $\cU_{q,k}(d)$, then  $\wh{\cZ}$  is smooth over  $\{q\}\times\aff^1$ by~\Ref{rmk}{gencone} and~\Ref{rmk}{verano}).

The composition of $\wh{\cZ}\to\cZ$ and the projection $\cZ\to\aff^1$ is a flat family  of surfaces $\varphi\colon\wh{\cZ}\to\aff^1$. The central fiber $\wh{Z}_0:=\varphi^{-1}(0)$ has normal crossings, it is  the union of $Y$ and the blow-up $\wt{P}$ of $P$ at the points of $X\cap Y\cap P$,    the curve $Y\cap P$ being glued to its strict transform in $\wt{P}$.  There will be an open dense  $B\subset\aff^1$ containing $0$  such that  $\wh{Z}_t:=\varphi^{-1}(t)$  is smooth for $t\in B\setminus\{0\}$, and it is  isomorphic to $Z_t:=V(g\cdot l+tf)$ in Case~(I), while it is  the blow-up of $Z_t$ at $q$ (an ODP) in Case~(II).
We replace $\wh{\cZ}$ by $ \varphi^{-1}(B)$ but we do not give it a new name.  

One proves that if $P$ is very general, then the following hold: 
\begin{enumerate}
\item[($\text{I}'$)]
In Case~(I), if $t$ is  very general in $B\setminus\{0\}$, then $\Pic(\wh{Z}_t)\otimes\QQ$ is generated by the classes of $\cO_{\wh{Z}_t}(1)$,  $\cO_{\wh{Z}_t}(C_1),\ldots,\cO_{\wh{Z}_t}(C_n)$. (Notice that $\wh{Z}_t=Z_t$ because we are in case~(I).)
\item[($\text{II}'$)]
In Case~(II), if $t$ is  very general in $B\setminus\{0\}$,  letting $\mu_t\colon \wh{Z}_t\to Z_t$ be the blow-up of $q$ and $R_t\subset\wh{Z}_t$ the exceptional curve, the group  $\Pic(\wh{Z}_t)\otimes\QQ$ is generated by the classes of $\mu_t^{*}\cO_{Z_t}(1)$,  $\mu_t^{*}\cO_{Z_t}(2C_1),\ldots,\mu_t^{*}\cO_{Z_t}(2C_n)$ and $\cO_{\wh{Z}_t}(R_t)$.
\end{enumerate}
One does this by controlling the Picard group of the degenerate fiber $\wh{Z}_0$. As proved in~\cite{grihar,lopez,brenol}  it suffices  to show that the following hold:
\begin{enumerate}
\item[(a)]
Let $\cV\subset|\cO_{\PP^3}(1)|$ be the open subset of planes intersecting transversely $C_0\cup\ldots\cup C_n$, let $I\subset (C_0\cup\ldots\cup C_n)\times \cV$ be the incidence subset and $\rho\colon I\to\cV$ be the natural finite map: then the mododromy of $\rho$ acts on a fiber $(D_0,\ldots, D_n, P)$ as the product of the symmetric groups $\gS_{\deg C_0}\times\ldots\times \gS_{\deg C_n}$. 
\item[(b)]
Let   $j\in\{0,\ldots,n\}$, let  $P\subset\PP^3$ be a very general plane, and let  $a,b\in C_j\cap P$ be distinct points; then the divisor 
class $a-b$ on the (smooth) curve $Y\cap P$ is not torsion.
\end{enumerate}
Now Item~(a) is Proposition II.2.6 of~\cite{lopez}. It remains to prove that~(b) holds. To this end we will show that $C_0$ is not planar and we will control the set of planes $P$ such that  $P\cap Y$ is reducible (see the proof of Item~(b) of Lemma 3.4 of~\cite{brenol}).
\begin{clm}
The curve $C_0$ (see~\eqref{decomposizione}) is not planar.
\end{clm}
\begin{proof}
By hypothesis $\pi^{*}\cO_{\PP^3}(d-3)(-E)$ is very ample, in particular it has a non-zero section, and hence there exists a non-zero  $\tau\in H^0(\PP^3,\cI_C(d-3))$.  
 Multiplying $\tau$ by sections of $\cO_{\PP^3}(3)$ we get that  that $h^0(\PP^3,\cI_C(d))\ge 20$. Now assume that $C_0$ is planar. Recall that $C=C_1\cup\ldots\cup C_n$, and let 
\begin{equation*}
H^0(\PP^3,\cI_C(d))\overset{\alpha}{\lra} H^0(Y,\cO_Y(d))
\end{equation*}
 be the restriction map. Since $(C+C_0)\in|\cO_Y(d)|$,  the image of $\alpha$ is equal to $H^0(Y,\cO_Y(C_0))$, and hence 
 has  dimension at most $4$ because $C_0$ is planar.  The kernel of $\alpha$ has dimension $4$ because $Y$ has degree $(d-1)$. It follows that $h^0(\PP^3,\cI_C(d))\le 8$, contradicting the inequality  $h^0(\PP^3,\cI_C(d))\ge 20$. 
\end{proof}
 Thus none of the curves $C_0,C_1,\ldots,C_n$ is planar. 
\begin{lmm}\label{lmm:pochi}
Let $Y\subset\PP^3$ be a surface which is either smooth or has ODP's. The set of planes $P$ such that $P\cap Y$ is reducible is the union of a finite set and the collection of pencils through lines of $Y$. 
\end{lmm}
\begin{proof}
Suppose the contrary. Then there exists a $1$-dimensional family of planes $P$ such that
 $P\cdot Y=C_1+C_2$ with $C_1,C_2$ divisors which intersect properly, $\supp C_1$ is irreducible,  and $\deg C_i>1$. Next, we distinguish between the two cases:
\begin{enumerate}
\item
 The generic $P$ does not contain any singular point of $Y$.
\item
The generic $P$  contains a single  point  $a\in \sing Y$, or two points $a,b\in \sing Y$.
\end{enumerate}
Assume that~(1) holds. Let $m_i:=\deg C_i$ for $i=1,2$. Then
\begin{equation}\label{echidna}
m_1 m_2=(C_1\cdot C_2)_P=(C_1\cdot C_2)_Y=(C_1\cdot(P-C_1))_Y=m_1-(C_1\cdot C_1)_Y
\end{equation}
where $(C_1\cdot C_2)_P$ is the intersection number of $C_1,C_2$ in the plane $P$, and $(C_1\cdot C_2)_Y$  is the intersection number of $C_1,C_2$ in the surface $Y$ (this makes sense because $Y$ has ODP singularities, and hence is $\QQ$-Cartier). The first equality  of~\eqref{echidna} holds by B\`ezout, the second equality is proved by a local computation of the multiplicity of intersection at each point of $C_1\cap C_2$ (one needs the hypothesis that $Y$ is smooth at each such point).   Thus~\eqref{echidna} gives $(C_1\cdot C_1)_Y=m_1(1-m_2)<0$, and this contradicts the hypothesis that $C_1$ moves in $Y$. If~(2) holds one argues similarly. We go through the computations  in the case that $P$ contains two singular points. Let $\wt{\PP}^3\to\PP^3$ be the blow up of  $\{a,b\}$, and $\wt{Y},\wt{P}\subset \wt{\PP}^3$ be the strict transforms of $Y$ and $P$ respectively.
By hypothesis $Y$ has an ODP at each of its singular points and hence $\wt{Y}$ is smooth, and of course $\wt{P}$ is smooth. Let $\wt{C}_i$ be the strict transform of $C_i$ in  $\wt{\PP}^3$.  Let $r_i:=\mult_{a} C_i$, $s_i:=\mult_{b} C_i$.  Then the equality 
\begin{equation}
(\wt{C}_1\cdot \wt{C}_2)_{\wt{P}}=(\wt{C}_1\cdot \wt{C}_2)_{\wt{Y}}
\end{equation}
gives  
\begin{equation}\label{anubi}
(\wt{C}_1\cdot \wt{C}_2)_{\wt{Y}}=-(m_1 m_2-m_1-r_1 r_2-s_1 s_2+r_1+s_1). 
\end{equation}
Now  $r_i+s_i\le m_i$ for $i=1,2$, because otherwise the line $\la a,b\ra$ would be contained in $Y\cap C_i$, and hence  we would be considering  curves residual to a line in $Y$, against the hypothesis. Since   $r_i+s_i\le m_i$ for $i=1,2$ the right-hand side of~\eqref{anubi} is strictly negative, and this is a contradiction.
\end{proof}
Now we prove that Item~(b) holds. Let $j\in\{0,\ldots,n\}$. Let  $a,b\in C_j$ be generic, in particular they are smooth  points of $Y$.
 By~\Ref{lmm}{pochi} every plane containing $a,b$ intersects $Y$ in an irreducible curve.  
Let $\wh{Y}\to Y$ be the blow-up of the base-locus of the pencil of plane sections of $Y$ containing $a,b$. Then $\wh{Y}$ has at most $A_n$-singularities, and hence is $\QQ$-factorial. Let $E,F$ be the exceptional sets over $a$ and $b$ respectively, both have strictly negative self-intersection. Let $i>0$ be such that $iE$ and $iF$ are Cartier. Let $\varphi\colon \wh{Y}\to\PP^1$ be the regular map defined by  the pencil of plane sections of $Y$ containing $a,b$; for $s\in\PP^1$ we let $D_s:=\varphi^{-1}(s)$. It suffices to prove that, given $r>0$, the set of $s\in\PP^1$ such that $\cO_{\wh{Y}}(riE-riF)|_{D_s}$ is trivial is finite.
 Assume the contrary: then, since   every plane containing $a,b$ intersects $Y$ in an irreducible curve, there exists  $\ell\in\QQ$ such that $riE-riF\equiv\varphi^{*}(\ell p)$ in $\Pic(\wh{Y})_{\QQ}$, where  $p\in\PP^1$ (see the proof of Item~(b) of Lemma 3.4 of~\cite{brenol}). It follows that the degrees of  $\cO_{\wh{Y}}(riE-riF)$ on $E$ and $F$ are both equal to $\ell$, and that is absurd because they have opposite signs.
 \qed
\section{Proof of the main result}\label{sec:chiudo}
We will prove~\Ref{thm}{zorba}. Let $Q\subset\PP^3$ be a smooth quadric and choose an isomorphism $\varphi\colon Q\overset{\sim}{\lra}\PP^1\times\PP^1$: we let $\cO_Q(a,b):=\varphi^{*}(\cO_{\PP^1}(a)\boxtimes\cO_{\PP^1}(b))$.  
\begin{prp}\label{prp:trereg}
A  curve in $|\cO_Q(2,3)|$ is $3$-regular. 
\end{prp}
\begin{proof}
Let  $D\in |\cO_Q(2,3)|$. 
Considering the exact sequence $0\to\cI_D\to\cO_{\PP^3}\to\cO_D\to 0$ we see right away that  if $i=2,3$, then  $H^i(\PP^3,\cI_D(3-i))=0$. In order to prove that $H^1(\PP^3,\cI_D(2))=0$ we must show that $H^0(\PP^3,\cO_{\PP^3}(2))\to H^0(C,\cO_{D}(2))$ is surjective. The map $H^0(\PP^3,\cO_{\PP^3}(2,2))\to H^0(Q,\cO_{Q}(2))$ is surjective, hence it suffices to prove that $H^0(Q,\cO_{Q}(2,2))\to H^0(C,\cO_{D}(2))$ is surjective. We have an exact sequence
\begin{equation*}
0\lra \cO_{Q}(0,-1)\lra \cO_{Q}(2,2)\lra \cO_D(2)\lra 0,
\end{equation*}
and since $H^1(Q,\cO_{Q}(0,-1))=0$ the map $H^0(Q,\cO_{Q}(2,2))\to H^0(D,\cO_{D}(2))$ is indeed surjective. 
\end{proof}
{\it Proof of~\Ref{thm}{zorba}.\/} If $d\le 6$ there is nothing to prove, hence we may assume that $d\ge 7$. 
Let $n:=\lfloor \frac{d-4}{3}\rfloor$. Choose disjoint smooth curves $C_1,\ldots,C_n$ such that each $C_j$ is a $(2,3)$-curve on a smooth quadric, and let $C:=C_1\cup\ldots\cup C_n$. We may assume that for $j\in\{1,\ldots,n\}$ the degree-$0$ class in $\CH_0(C_j)$ given by  $5 c_1(K_{C_j})-2 c_1(\cO_{C_j}(1))$ is \emph{not} zero. Let us show that the hypotheses of~\Ref{thm}{peppapig} are satisfied. Let $j\in\{1,\ldots,n\}$. We let $\pi_j\colon W_j\to\PP^3$ be the blow-up of $C_j$, and  $F_j\subset W_j$ be the exceptional divisor. Then $\pi_j^{*}\cO_{\PP^3}(3)(-F_j)$ is globally generated, and  $\pi_j^{*}\cO_{\PP^3}(4)(-F_j)$ is very ample: since $d-3\ge 3(n-1)+4$ it follows that $\pi^{*}\cO_{\PP^3}(d-3)(-E)$ is very ample. Let $j\in\{1,\ldots,n\}$: since  $d\ge 7$ the cohomology group $H^1(C_j,T_{C_j}(d-4))$ vanishes, and hence $H^1(C,T_{C}(d-4))=0$. By~\Ref{prp}{trereg} and
Example 1.8.32 of~\cite{roblaz} the curve $C$ is $3n$-regular, and since $3n\le(d-4)$ the curve $C$ is $(d-2)$-regular. Lastly, by construction no curve $C_j$ is planar. We have shown that the hypotheses of~\Ref{thm}{peppapig} are satisfied, and hence~\Ref{hyp}{ipnole} holds for $H\in |\cO_{\PP^3}(d)|$. Let $X\in|\cI_C(d)|$ be smooth and very generic: since  for $j\in\{1,\ldots,n\}$ the class $5 c_1(K_{C_j})-2 c_1(\cO_{C_j}(1))$ is not zero, the decomposable classes  $H^2,C_1^2,\ldots,,C_n^2$ on $X$ are linearly independent
by~\Ref{prp}{vaccini}. Thus $\DCH_0(X)$ has rank at least $n+1=\lfloor \frac{d-1}{3}\rfloor$. 
\qed

\end{document}